\documentclass[11pt,reqno,sumlimits]{amsart}
\usepackage{amsfonts, amsmath, amscd, amssymb, euscript, amsthm, array, booktabs,
dcolumn, shortvrb, tabularx, units, url, mathrsfs}
\usepackage[pdftex]{graphicx}
\PassOptionsToPackage{bookmarks=false}{hyperref}
\usepackage[pdfstartview=FitH]{hyperref}
\usepackage[all]{xy}
\normalsize

\DeclareMathOperator{\op}{op}

\DeclareMathOperator{\re}{Re}
\DeclareMathOperator{\CCS}{CCS}
\DeclareMathOperator{\Aut}{Aut}
\DeclareMathOperator{\LL}{L}

\DeclareMathOperator{\todd}{Todd}
\DeclareMathOperator{\an}{a}
\DeclareMathOperator{\FL}{FL}

\DeclareMathOperator{\dr}{dR}

\DeclareMathOperator{\ind}{ind}

\DeclareMathOperator{\spin}{spin}

\DeclareMathOperator{\CS}{CS}

\DeclareMathOperator{\odd}{odd}
\DeclareMathOperator{\even}{even}

\DeclareMathOperator{\im}{Im}
\DeclareMathOperator{\id}{id}
\DeclareMathOperator{\ch}{ch}
\DeclareMathOperator{\GL}{GL}

\DeclareMathOperator{\End}{End}
\DeclareMathOperator{\ho}{Hom}
\DeclareMathOperator{\rk}{rank}
\DeclareMathOperator{\str}{str}
\DeclareMathOperator{\tr}{tr}
\begin{document}
\setlength{\baselineskip}{1.5\baselineskip}
\theoremstyle{definition}
\newtheorem{coro}{Corollary}
\newtheorem*{note}{Note}
\newtheorem{thm}{Theorem}
\newtheorem*{ax}{Axiom}
\newtheorem{defi}{Definition}
\newtheorem{lemma}{Lemma}
\newtheorem{ass}{Assumption}
\newtheorem*{claim}{Claim}
\newtheorem{exam}{Example}
\newtheorem{prop}{Proposition}
\newtheorem{remark}{Remark}
\newcommand{\wt}[1]{{\widetilde{#1}}}
\newcommand{\ov}[1]{{\overline{#1}}}
\newcommand{\un}[1]{{\underline{#1}}}
\newcommand{\wh}[1]{{\widehat{#1}}}
\newcommand{\poin}{Poincar$\acute{\textrm{e }}$}
\newcommand{\deff}[1]{{\bf\emph{#1}}}
\newcommand{\boo}[1]{\boldsymbol{#1}}
\newcommand{\abs}[1]{\lvert#1\rvert}
\newcommand{\norm}[1]{\lVert#1\rVert}
\newcommand{\inner}[1]{\langle#1\rangle}
\newcommand{\poisson}[1]{\{#1\}}
\newcommand{\biginner}[1]{\Big\langle#1\Big\rangle}
\newcommand{\set}[1]{\{#1\}}
\newcommand{\Bigset}[1]{\Big\{#1\Big\}}
\newcommand{\BBigset}[1]{\bigg\{#1\bigg\}}
\newcommand{\dis}[1]{$\displaystyle#1$}
\newcommand{\R}{\mathbb{R}}
\newcommand{\EE}{\mathbb{E}}
\newcommand{\GG}{\mathbb{G}}
\newcommand{\N}{\mathbb{N}}
\newcommand{\Z}{\mathbb{Z}}
\newcommand{\Q}{\mathbb{Q}}
\newcommand{\E}{\mathcal{E}}
\newcommand{\T}{\mathcal{T}}
\newcommand{\G}{\mathcal{G}}
\newcommand{\F}{\mathcal{F}}
\newcommand{\I}{\mathcal{I}}
\newcommand{\V}{\mathcal{V}}
\newcommand{\W}{\mathcal{W}}
\newcommand{\SSS}{\mathcal{S}}
\newcommand{\h}{\mathbb{H}}
\newcommand{\C}{\mathbb{C}}
\newcommand{\A}{\mathcal{A}}
\newcommand{\LLL}{\mathcal{L}}
\newcommand{\HH}{\mathcal{H}}
\newcommand{\D}{\mathcal{D}}
\newcommand{\PP}{\mathcal{P}}
\newcommand{\K}{\mathcal{K}}
\newcommand{\RRR}{\mathscr{R}}
\newcommand{\AAA}{\mathscr{A}}
\newcommand{\DDD}{\mathscr{D}}
\newcommand{\e}{\mathscr{E}}
\newcommand{\w}{\mathscr{W}}
\newcommand{\f}{\mathscr{F}}
\newcommand{\z}{\mathcal{Z}}
\newcommand{\g}{\mathscr{G}}
\newcommand{\so}{\mathfrak{so}}
\newcommand{\gl}{\mathfrak{gl}}
\newcommand{\aaa}{\mathbb{A}}
\newcommand{\bbb}{\mathbb{B}}
\newcommand{\ttt}{\mathbb{T}}
\newcommand{\DD}{\mathsf{D}}
\newcommand{\ff}{\mathsf{F}}
\newcommand{\FF}{\mathbb{F}}
\newcommand{\ccc}{\bold{c}}
\newcommand{\sss}{\mathbb{S}}
\newcommand{\cdd}[1]{\[\begin{CD}#1\end{CD}\]}
\numberwithin{equation}{subsection}
\normalsize
\title[RRG theorem]{Local index theory and the Riemann--Roch--Grothendieck theorem for
complex flat vector bundles}
\author{Man-Ho Ho}
\address{Alumni of Boston University \\ Hong Kong}
\email{homanho@bu.edu}
\maketitle
\nocite{*}
\begin{abstract}
The purpose of this paper is to give a proof of the real part of the
Riemann--Roch--Grothendieck theorem for complex flat vector bundles at the differential
form level in the even dimensional fiber case. The proof is, roughly speaking, an
application of the local family index theorem for a perturbed twisted spin Dirac operator,
a variational formula of the Bismut--Cheeger eta form without the kernel bundle assumption
in the even dimensional fiber case, and some properties of the Cheeger--Chern--Simons class
of complex flat vector bundle.
\end{abstract}
\tableofcontents
\section{Introduction}

Let $\pi:X\to B$ be a submersion with closed fibers $Z$ and $F\to X$ a complex flat vector
bundle with flat connection $\nabla^F$. The Riemann--Roch--Grothendieck (RRG) theorem
(\ref{eq 1.0.1}) for complex flat vector bundles is an equality in $H^{\odd}(B; \C/\Q)$
stating that
\begin{equation}\label{eq 1.0.1}
\CCS(H(Z, F|_Z), \nabla^{H(Z, F|_Z)}))=\int_{X/B}e(T^VX)\cup\CCS(F, \nabla^F),
\end{equation}
where $\CCS(F, \nabla^F)$ is the Cheeger--Chern--Simons class of $(F, \nabla^F)$
\cite[(1.1)]{MZ08} and $H(Z, F|_Z)\to B$ is the cohomology bundle with flat connection
$\nabla^{H(Z, F|_Z)}$. Both sides of (\ref{eq 1.0.1}) are defined in terms of the mod $\Q$
reduction of the de Rham class of certain closed odd differential forms.

Bismut--Lott prove the imaginary part of (\ref{eq 1.0.1}) at the differential form level
\cite[Theorem 3.23]{BL95}. The real part of (\ref{eq 1.0.1}), referred as the real RRG
theorem, is an equality in $H^{\odd}(X; \R/\Q)$ stating that
\begin{equation}\label{eq 1.0.2}
\re(\CCS(H(Z, F|_Z), \nabla^{H(Z, F|_Z)}))=\int_{X/B}e(T^VX)\cup\re(\CCS(F, \nabla^F)).
\end{equation}
(\ref{eq 1.0.2}) is first proved by Bismut under the assumption that the fibers are
fiberwise orientable \cite[Theorem 3.2]{B05}, and later proved by Ma--Zhang in full
generality \cite[Theorem 1.1]{MZ08}.

The main result of this paper is a $\Z_2$-graded version of (\ref{eq 1.0.2}) at the
differential form level for $\dim(Z)$ even (Theorem \ref{thm 3.3.1}), i.e. the complex
flat vector bundle and its flat connection $(F, \nabla^F)$ in (\ref{eq 1.0.2}) are
$\Z_2$-graded and $F\to X$ has virtual rank zero. By taking an appropriate $(F, \nabla^F)$
in Theorem \ref{thm 3.3.1} we recover a result by Ma--Zhang \cite[(3.98)]{MZ08} for
$\dim(Z)$ even. By arguing as in \cite[p.614]{MZ08}, which makes use of a result by Bismut
\cite[Theorem 3.12]{B05}, we obtain (\ref{eq 1.0.2}) for $\dim(Z)$ even. Along the way we
prove a variational formula of the Bismut--Cheeger eta form without the kernel bundle
assumption, which could be of independent interest. A brief description of the main results
is given in Section \ref{s 1.2}.

The proof of (\ref{eq 1.0.2}) by Bismut \cite{B05} makes critical use of adiabatic limit
computations of the reduced $\eta$-invariant of certain Dirac operators and Cheeger--Simons'
geometric index theorem \cite[Theorem 9.2]{CS85}. On the other hand, the proof of
(\ref{eq 1.0.2}) by Ma--Zhang \cite{MZ08} uses adiabatic limit computations of the reduced
$\eta$-invariant of the sub-signature operator developed by Zhang \cite{Z04}.

Our proof of Theorem \ref{thm 3.3.1} makes use of the local family index theorem (local
FIT) for a perturbed twisted spin Dirac operator and some properties of the
Cheeger--Chern--Simons class, and does not involve the reduced $\eta$-invariant nor
adiabatic limit calculations. However, since the result by Bismut \cite[Theorem 3.12]{B05}
we use to derive (\ref{eq 1.0.2}) for $\dim(Z)$ even is proved by using the reduced
$\eta$-invariant and adiabatic limit calculations, our proof of (\ref{eq 1.0.2}) for
$\dim(Z)$ even still uses these tools.

Since for $\dim(Z)$ odd, the right-hand side of (\ref{eq 1.0.2}) is zero, the odd
dimensional fiber case of the real RRG theorem states that the left-hand side of
(\ref{eq 1.0.2}) is zero. It is well known that most of the local family index type
theorems in the odd dimensional fiber case can be proved by applying a trick due to
Bismut--Freed (see the proof of \cite[Theorem 2.10]{BF86} and also \cite[\S9]{FL10}) and
the corresponding local FIT in the even dimensional fiber case. However, we are unable to
give another proof of the odd dimensional fiber case of (\ref{eq 1.0.2}) at this moment.
The reason seems to be the incompatibility of our techniques and the trick by
Bismut--Freed. However, we have an idea to overcome the incompatibility, and the same idea
can also be applied to give another proof of the odd dimensional fiber case of the
Grothendieck--Riemann--Roch theorem in flat $K$-theory (flat GRR theorem) \cite{L94} at
the differential form level. These questions will be treated in a future paper.

\subsection{Method of proof}\label{s 1.2}

In this subsection we describe the main results in this paper and outline the method of
proof.

The motivation of proving Theorem \ref{thm 3.3.1} comes from a fundamental but crucial
observation by Ma--Zhang \cite[(2.44)]{MZ08} that for a complex flat vector bundle
$F\to X$ with flat connection $\nabla^F$ equipped with a Hermitian metric $g^F$, where
$\nabla^F$ is not assumed to be unitary with respect to $g^F$, the real part of the
Cheeger--Chern--Simons class $\CCS(F, \nabla^F)$ is given by
\begin{equation}\label{eq 1.2.1}
\re(\CCS(F, \nabla^F))=\bigg[\frac{1}{k}\CS(\nabla^{kF}_0, k\nabla^{F, u})\bigg]\mod\Q
\in H^{\odd}(X; \R/\Q),
\end{equation}
where $k\in\N$ is such that $kF\cong k\C^{\rk(F)}$ as smooth complex vector bundles,
$\nabla^{F, u}$ is a unitary connection on $F\to X$ constructed out of $\nabla^F$ and
$\nabla^{kF}_0$ is a trivial connection on $kF\to X$. The details will be given in Section
\ref{s 2.3}. By considering $\F:=(F, g^F, \nabla^{F, u}, 0)$ as a generator of the flat
$K$-group $K^{-1}_{\LL}(X)$ \cite[Definition 5]{L94}, (\ref{eq 1.2.1}) can be written as
\begin{equation}\label{eq 1.2.2}
\re(\CCS(F, \nabla^F))=-\ch_{\R/\Q}(\F),
\end{equation}
where $\ch_{\R/\Q}:K^{-1}_{\LL}(X)\to H^{\odd}(X; \R/\Q)$ is the flat Chern character
\cite[Definition 9]{L94}. On the other hand, given a submersion $\pi:X\to B$ with closed,
oriented and spin$^c$ fibers and a $\Z_2$-graded generator $\E$ of $K^{-1}_{\LL}(X)$, the
flat GRR theorem \cite[Corollary 4]{L94} is an equality in $H^{\odd}(B; \R/\Q)$ stating
that
\begin{equation}\label{eq 1.2.3}
\ch_{\R/\Q}(\ind^{\an}_{\LL}(\E))=\int_{X/B}\todd(X/B)\cup\ch_{\R/\Q}(\E),
\end{equation}
where $\ind^{\an}_{\LL}:K^{-1}_{\LL}(X)\to K^{-1}_{\LL}(B)$ is the analytic index in flat
$K$-theory \cite[Definition 14]{L94}. By comparing (\ref{eq 1.0.2}) and (\ref{eq 1.2.3})
using (\ref{eq 1.2.2}) we wonder if the real RRG theorem can be proved in the same way as
the flat GRR theorem given in \cite{H16}.

Since the main idea of the proof of Theorem \ref{thm 3.3.1} is similar to that of the flat
GRR theorem given in \cite{H16}, we briefly recall it here. Given the setup of the flat
GRR theorem described above, consider the associated submersion $\wt{\pi}:\wt{X}\to\wt{B}$,
where $I=[0, 1]$, $\wt{X}=X\times I$ (similarly for $\wt{B}$) and $\wt{\pi}:=\pi\times\id$.
For any manifold $X$, define a map $i_{X, k}:X\to\wt{X}$ by $i_{X, k}(x)=(x, k)$. Given a
$\Z_2$-graded generator $(E, g^E, \nabla^E, \phi)$ of $K^{-1}_{\LL}(X)$, we construct a
complex vector bundle $\e\to\wt{X}$ with a Hermitian metric $g^{\e}$ and a unitary
connection $\nabla^{\e}$ such that $i_{X, 1}^*\nabla^{\e}=m\nabla^{E^+}$ and $i_{X, 0}^*
\nabla^{\e}=m\nabla^{E^-}$ for some $m\in\N$ satisfying $mE^+\cong mE^-$. Assume the family
of kernels $\ker(\DD^{S^c\otimes\e}_{\wt{b}})$ of the twisted spin$^c$ Dirac operator
$\DD^{S^c\otimes\e}$ parameterized by $\wt{B}$ form a ($\Z_2$-graded) complex vector
bundle. The local FIT for $\DD^{S^c\otimes\e}$ is given by
\begin{equation}\label{eq 1.2.4}
d\wt{\eta}=\int_{\wt{X}/\wt{B}}\todd(\nabla^{S^c(T^V\wt{X})})\wedge\ch(\nabla^{\e})-
\ch(\nabla^{\ker(\DD^{S^c\otimes\e})}),
\end{equation}
where $\wt{\eta}$ is the Bismut--Cheeger eta form associated to $\DD^{S^c\otimes\e}$. By
integrating (\ref{eq 1.2.4}) along the fibers of the trivial fibration $\wt{B}\to B$ we
obtain an equality of closed odd differential forms refining (\ref{eq 1.2.3}).

At a first glance, one may suspect that the above strategy can be directly applied to
prove Theorem \ref{thm 3.3.1} if the spin$^c$ Dirac operator is replaced by the de Rham
operator. However, the above strategy causes two problems in the current situation.

For the first problem, recall that the local FIT for the twisted de Rham operator
$\DD^{\wt{Z}, \dr}$ for $\wt{\pi}:\wt{X}\to\wt{B}$ \cite[\S3(c)]{B05} states that
\begin{equation}\label{eq 1.2.5}
d\wt{\eta}^{\dr}=\int_{\wt{X}/\wt{B}}e(\nabla^{T^V\wt{X}})\wedge\ch(\nabla^{\f, u})-
\ch(\nabla^{H(\wt{Z}, \f|_{\wt{Z}}), u}),
\end{equation}
where $\f\to\wt{X}$ is a complex flat vector bundle with flat connection $\nabla^{\f}$,
and $\wt{\eta}^{\dr}$ is the Bismut--Cheeger eta form associated to $\DD^{\wt{Z}, \dr}$.
Since $\ch(\nabla^{\f, u})=\rk(\f)$ and
$$\ch(\nabla^{H(\wt{Z}, \f|_{\wt{Z}}), u})=\rk(H(\wt{Z}, \f|_{\wt{Z}}))=\rk(\f)\chi
(\wt{Z}),$$
the right-hand side of (\ref{eq 1.2.5}) is zero, and therefore $\wt{\eta}^{\dr}$ is closed.
A stronger result by Bismut \cite[Theorem 3.7]{B05} states that $\wt{\eta}^{\dr}=0$. Thus
integrating (\ref{eq 1.2.5}) along the fibers of $\wt{B}\to B$ gives $0=0$ instead of an
equality of closed odd differential forms refining (\ref{eq 1.0.2}).

The second problem is the assumption of the existence of the kernel bundle
$\ker(\DD^{S^c\otimes\e})\to\wt{B}$ in our proof of the flat GRR theorem. It is well known
that if the kernel bundle does not exist, then one can perturb the Dirac operator so that
the resulting family of kernels form a vector bundle. Thus the statement can usually be
reduced to the kernel bundle case. This reduction process can be applied to the flat GRR
theorem, but not to the real RRG theorem. For the flat GRR theorem, if the kernel bundle
exists then the bundle part of the analytic index in flat $K$-theory is defined to be the
kernel bundle; otherwise the bundle part is defined to be any fixed choice of a finite
rank subbundle $L\to B$ in the approach by Mi\v s\v cenko--Fomenko, which is outlined
below. Thus the flat GRR theorem at the differential form level can be proved by reducing
it to the kernel bundle case. For Theorem \ref{thm 3.3.1}, however, there is a specific
''target" regardless of the existence of the kernel bundle of a suitable Dirac operator,
namely, the cohomology bundle $H(Z, F|_Z)\to B$. If we prove Theorem \ref{thm 3.3.1} under
the kernel bundle assumption, whose existence is actually unknown in reality, we would
have to pull back, for example, a unitary connection on the possibly non-existing kernel
bundle to $H(Z, F|_Z)\to B$ (see (\ref{eq 1.2.7})). This argument is certainly incorrect.
We would like to thank the referee for pointing this out.

Before we outline the solutions to these two problems, let us briefly recall that there
are (at least) two approaches to deal with the non-existence of the kernel bundle.
\begin{itemize}
  \item One approach is given by Atiyah--Singer \cite{AS71} (see also \cite[\S9.5]{BGV}).
        The idea is to find a trivial bundle $\C^N\to B$ and a linear map $s:\C^N\to
        (\pi_*E)^-$ in order to perturb the Dirac operator $\DD$ by a smoothing operator
        $R_s$ induced by $s$, so that $(\DD+R_s)_+:(\pi_*E)^+\oplus\C^N\to(\pi_*E)^-$ is
        surjective, and therefore $\ker((\DD+R_s)_-)=0$. In this case the $K$-theoretic
        analytic index of $[E]\in K(X)$ is defined to be $[\ker(\DD+R_s)]-[\C^N]$.
  \item Another approach is given by Mi\v s\v cenko--Fomenko \cite[Lemma 2.2]{MF79}. The
        idea is to find a $\Z_2$-graded finite rank subbundle $L\to B$ of $\pi_*E\to B$
        and a $\Z_2$-graded complementary subbundle $K\to B$ such that $\DD_+$ is block
        diagonal with respect to the decomposition $(\pi_*E)^\pm=K^\pm\oplus L^\pm$ and
        it restricts to an isomorphism $K^+\to K^-$. In this case the $K$-theoretic
        analytic index of $[E]\in K(X)$ is defined to be $[L^+]-[L^-]$.
\end{itemize}

The idea of solving the above problems is, roughly speaking, to consider the following
perturbed twisted spin Dirac operator
\begin{equation}\label{eq 1.2.6}
\DD^{S\wh{\otimes}(S^*\otimes\f)}+\wt{V}\in\Gamma(\wt{X}, \End^-(S(T^V\wt{X})\wh{\otimes}
(S(T^V\wt{X})^*\otimes\f))),
\end{equation}
where the details will be given in Section \ref{s 2.3}. The unitary connection
$\nabla^{\f, u}$ on $\f\to\wt{X}$ defining $\DD^{S\wh{\otimes}S^*\otimes\f}$ is chosen in
the way that its curvature satisfies a certain condition (namely, (\ref{eq 2.2.4})) only
on $(S(T^V\wt{X})^*\otimes\f)|_{\partial\wt{X}}\to\partial\wt{X}$. The reason of
perturbing and twisting the spin Dirac operator as in (\ref{eq 1.2.6}) is due to
$$S(T^VX)\wh{\otimes}S(T^VX)^*\cong\Lambda(T^VX)^*\otimes\C$$
as $\Z_2$-graded complex vector bundles and the following result by Bismut--Zhang
\cite[Proposition 4.12]{BZ92}:
$$\DD^{\Lambda\otimes F}+V=\DD^{Z, \dr},$$
where the curvature of the unitary connection $\nabla^{F, u}$ defining $\DD^{\Lambda
\otimes F}$ satisfies (\ref{eq 2.2.4}). If we assumed the existence of the kernel bundle of
the Dirac operator (\ref{eq 1.2.6}), then one could prove that
$$\ker(\DD^{S\wh{\otimes}(S^*\otimes\f)}+\wt{V})\cong\ker(\DD^{\Lambda\otimes\f}+\wt{V})$$
as $\Z_2$-graded complex vector bundles. This would imply
\begin{equation}\label{eq 1.2.7}
i_{B, k}^*\ker(\DD^{S\wh{\otimes}(S^*\otimes\f)}+\wt{V})\cong i_{B, k}^*\ker(\DD^{\Lambda
\otimes\f}+\wt{V})\cong i_{B, k}^*\ker(\DD^{\wt{Z}, \dr})
\end{equation}
as $\Z_2$-graded complex vector bundles over $B$ for $k\in\set{0, 1}$. However, our
assumption is not verified in general, so we need to adopt one of the two aforementioned
approaches.

We choose the approach given by Mi\v s\v cenko--Fomenko, and the corresponding results of
the local FIT in this approach is given by Freed--Lott \cite[\S7]{FL10}. More precisely,
we first establish a variational formula of the Bismut--Cheeger eta form without the
kernel bundle assumption (Proposition \ref{prop 3.1.1}). As a byproduct we give another
proofs of the facts that
\begin{itemize}
  \item the analytic index in differential $K$-theory defined without the kernel bundle
        assumption does not depend on the choice of the finite rank subbundle $L\to B$
        (Corollary \ref{coro 3.1.1}),
  \item if the kernel bundle exists then the two definitions of the analytic index in
        differential $K$-theory, under and without the kernel bundle assumption, coincide
        as elements in the differential $K$-group (Corollary \ref{coro 3.1.2}).
\end{itemize}

These results are first proved by Freed--Lott \cite[(3) and (4) of Corollary 7.36]{FL10} as
a consequence of the FIT in differential $K$-theory \cite[Theorem 7.35]{FL10}. Our proofs
of these results are an application of Proposition \ref{prop 3.1.1}, and do not make use
of the FIT in differential $K$-theory. Note that these results are stated in terms of spin
fibers in this paper, as opposed to \cite{FL10}, which are stated in terms of $\spin^c$
fibers. By a minor modification all the results in Section \ref{s 3.1} can be extended to
$\spin^c$ fibers case.

The idea of proving Theorem \ref{thm 3.3.1} follows closely to the proof of Proposition
\ref{prop 3.1.1}. In more detail, given a $\Z_2$-graded complex flat vector bundle $F\to X$
of virtual rank zero with $\Z_2$-graded flat connection $\nabla^F$, we take $k\in\N$ large
enough so that there exist smooth bundle isomorphisms $j:kF^+\to kF^-$ and
$h^\bullet_\pm:kH^\bullet(Z, F^\pm|_Z)\cong\C^{kn^\bullet_\pm}$ for $\bullet\in\set{\even,
\odd}$, where $n^\bullet_\pm=\rk(H^\bullet(Z, F^\pm|_Z))$. Then by putting a $\Z_2$-graded
Hermitian metric $g^F$ on $F\to X$, there exists a smooth isometric isomorphism
$\wt{j}:kF^+\to kF^-$ associated to $(k, j, g^F)$, i.e. $kg^+=\wt{j}^*kg^-$. Note that
$k\nabla^+$ and $\wt{j}^*k\nabla^-$ are two flat connections on $kF^+\to X$. By joining
$k\nabla^+$ and $\wt{j}^*k\nabla^-$ by a smooth path of connections, one can construct a
Hermitian bundle $\f\to\wt{X}$ with a unitary connection $\nabla^{\f, u}$ which pulls back
to two unitary connections on $kF^+\to X$ associated to $k\nabla^+$ and $\wt{j}^*k\nabla^-$,
respectively. By considering the perturbed twisted spin Dirac operator (\ref{eq 1.2.6})
without assuming the existence of its kernel bundle, one obtains an analog of local FIT
that is suitable for our purpose. We then integrate the so obtained equality of closed
differential forms along the fibers of the trivial fibration $\wt{B}\to B$. Since the
kernel bundle exists over $\partial\wt{B}$, the idea is to ``replace" the pullback of the
$\Z_2$-graded finite rank subbundle $\LLL\to\wt{B}$ to $\partial\wt{B}$ by the kernel
bundle. At this step we cannot directly apply Corollary \ref{coro 3.1.2}, but it provides
a guidance on what and how the geometric data should be replaced. The reason that Corollary
\ref{coro 3.1.2} is not directly applicable here is due to the perturbed twisted spin Dirac
operator (\ref{eq 1.2.6}). The definition of the analytic index in differential $K$-theory
without the kernel bundle assumption does not allow us to perturb $\DD^{S\wh{\otimes}
(S^*\otimes\f)}$ by $\wt{V}$. Thus we are forced to mimic the proof of Proposition
\ref{prop 3.1.1} when proving Theorem \ref{thm 3.3.1}. During the process of ''replacing"
the kernel bundle over $\partial\wt{B}$ the isomorphism
$$\wt{j}:(kF^+, \wt{j}^*k\nabla^-)\to(kF^-, k\nabla^-)$$
of complex flat vector bundles plays an important role. The proof of Theorem \ref{thm 3.3.1}
is completed by mimicking the proof of Proposition \ref{prop 3.1.1} adapted to the current
situation. By taking an appropriate $\Z_2$-graded complex flat vector bundle with
$\Z_2$-graded flat connection in Theorem \ref{thm 3.3.1}, and using the properties of the
Cheeger--Chern--Simons class established in Section \ref{s 3.2} and a result by Bismut
\cite[Theorem 3.12]{B05}, we deduce (\ref{eq 1.0.2}) for $\dim(Z)$ even.

We would like to emphasize that the use of differential $K$-theory in this paper is not
absolutely necessary (we could state Corollary \ref{coro 3.1.1} and Corollary
\ref{coro 3.1.2} without it). However, differential $K$-theory is a very convenient tool
to keep track of the changes of the geometric objects in local index theory when the
defining data are deformed, and it effectively shortens the presentation.

\subsection{Outline}

The paper is organized as follows. In Section 2 we review the background material,
including some aspects of the Chern character form and the Chern--Simons form, the
Cheeger--Chern--Simons class of complex flat vector bundles, the setup and the statement
of the local FIT for twisted de Rham operator, and the local FIT for twisted spin Dirac
operator under and without the kernel bundle assumption. In Section 3 we prove the main
results in this paper. In Section \ref{s 3.1} we prove a variational formula of the
Bismut--Cheeger eta form without the kernel bundle assumption in the even dimensional fiber
case. Then we show the independence of the choice of the $\Z_2$-graded finite rank subbundle
in the definition of the analytic index in differential $K$-theory without the kernel bundle
assumption, and we prove that if the kernel bundle exists then the two definitions of the
analytic index in differential $K$-theory coincide. Section \ref{s 3.2} is devoted to some
basic properties of the Cheeger--Chern--Simons class that will be used in Section \ref{s 3.3}.
In Section \ref{s 3.3} we prove the main result of the paper. First of all we prove a
$\Z_2$-graded version of (\ref{eq 1.0.2}) for $\dim(Z)$ even at the differential form level.
Then we deduce (\ref{eq 1.0.2}) for $\dim(Z)$ even.

\section*{Acknowledgement}

We would like to thank Steve Rosenberg for many helpful advice and his constant
encouragement, Sebastian Goette for generously answering many (awkward) questions in
MathOverflow, and Bo Liu for a discussion on a variational formula of the Bismut--Cheeger
eta form. Last but not least we would like to thank the referee(s) for his/her patience,
careful reading and helpful comments, correcting the mistakes and improving the presentation
of the paper significantly.

\section{Background material}

In this paper $X$ and $B$ are closed manifolds and $I$ is the closed interval $[0, 1]$.
Given a manifold $X$, define $\wt{X}=X\times I$. Given $t\in[0, 1]$, define a map
$i_{X, t}:X\to\wt{X}$ by $i_{X, t}(x)=(x, t)$. Denote by $p_X:\wt{X}\to X$ the standard
projection map. For $k\geq 0$, denote by $\Omega^k_\Q(X; \C)$ the set of all
complex-valued closed $k$-forms on $X$ with periods in $\Q$, and write $\Omega^k_\Q(X)$
for $\Omega^k_\Q(X; \R)$.

Let $E\to X$ and $F\to X$ be complex vector bundles, where $E\to X$ is $\Z_2$-graded.
Denote by $E^{\op}\to X$ the $\Z_2$-graded complex vector bundle whose $\Z_2$-grading is
the opposite of $E\to X$, i.e. $(E^{\op})^+=E^-$ and $(E^{\op})^-=E^+$. We will also use
the notation $\op$ for other $\Z_2$-graded objects. Denote by $E\otimes F\to X$ the
$\Z_2$-graded tensor product if $F\to X$ is ungraded; and by $E\wh{\otimes}F\to X$ the
$\Z_2$-graded tensor product if $F\to X$ is $\Z_2$-graded.

\subsection{Chern character form and Chern--Simons form}\label{s 2.1}

In this subsection we recall the definitions of the Chern character form and the
Chern--Simons form, and also fix the sign convention.

Let $E\to X$ be a complex vector bundle with a connection $\nabla^E$. The Chern
character form of $\nabla^E$ is defined by
$$\ch(\nabla^E)=\tr(e^{-\frac{1}{2\pi i}(\nabla^E)^2})\in\Omega^{\even}_\Q(X; \C).$$

There is a ``canonical" transgression form \dis{\CS(\nabla^E_0, \nabla^E_1)\in
\frac{\Omega^{\odd}(X; \C)}{\im(d)}} between the Chern character forms of two
connections in the sense that
\begin{equation}\label{eq 2.1.1}
d\CS(\nabla^E_0, \nabla^E_1)=\ch(\nabla^E_1)-\ch(\nabla^E_0).
\end{equation}
One of the definitions of $\CS(\nabla^E_1, \nabla^E_0)$ is given as follows. In the
following $k\in\set{0, 1}$ is fixed. Let $\e\to\wt{X}$ be a complex vector bundle
with a connection $\nabla^{\e}$. Since $p_X\circ i_{X, k}=\id_X$ and $i_{X, k}\circ
p_X\sim\id_{\wt{X}}$, it follows that $\e\cong p_X^*(i_{X, k}^*\e)$. Thus
\begin{equation}\label{eq 2.1.2}
E_0:=i_{X, 0}^*\e\cong i_{X, 0}^*p_X^*(i_{X, 0}^*\e)\cong i_{X, 1}^*p_X^*(i_{X, 0}^*\e)
\cong i_{X, 1}^*\e=:E_1.
\end{equation}
Write $E=E_0\cong E_1$. For a fixed $k\in\set{0, 1}$, define a connection $\nabla^E_k$
on $E\to X$ by
$$\nabla_k^E:=i_{X, k}^*\nabla^{\e}.$$
The Chern--Simons form $\CS(\nabla^E_0, \nabla^E_1)$ is defined to be
\begin{equation}\label{eq 2.1.3}
\CS(\nabla^E_0, \nabla^E_1)=-\int_{\wt{X}/X}\ch(\nabla^{\e})\mod\im(d),
\end{equation}
where $\wt{X}/X$ denotes the fiber of the fiber bundle $\wt{X}\to X$, and
\dis{\int_{\wt{X}/X}} denotes the integration along the fiber.

We will need some facts about integration along the fibers. Let $\pi:M\to B$ be a smooth
fiber bundle with compact fibers. By \cite[Chapter 1]{BT82}, we have
$$\int_{M/B}\pi^*\alpha\wedge\beta=\alpha\wedge\bigg(\int_{M/B}\beta\bigg)$$
for all $\alpha\in\Omega(B)$ and $\beta\in\Omega(M)$. If $M$ is a manifold with boundary
and the fiber bundle $\pi:M\to B$, whose fibers are compact and of dimension $n$,
satisfies certain orientability assumptions, then Stokes' theorem for integration along
the fibers (see, for example, \cite[(1.52)]{BM06} and \cite[Problem 4 (p.311)]{GHV}) is
given by
\begin{equation}\label{eq 2.1.4}
(-1)^{k-n+1}\int_{\partial M/B}i^*\omega=\int_{M/B}d^M\omega-d^B\int_{M/B}\omega,
\end{equation}
where $i:\partial M\to M$ is the inclusion map and $\omega\in\Omega^k(M)$.

By considering the fiber bundle $\wt{X}\to X$ and taking $\omega=\ch(\nabla^{\e})$ in
(\ref{eq 2.1.4}), we have
\begin{displaymath}
\begin{split}
d\CS(\nabla^E_0, \nabla^E_1)&=-d\int_{\wt{X}/X}\ch(\nabla^{\e})=-\int_{\wt{X}/X}d\ch
(\nabla^{\e})+\int_{\partial\wt{X}/X}i^*\ch(\nabla^{\e})\\
&=\ch(\nabla_1^E)-\ch(\nabla_0^E).
\end{split}
\end{displaymath}
Thus the Chern--Simons form defined by (\ref{eq 2.1.3}) satisfies (\ref{eq 2.1.1}).

Given two connections $\nabla^E_1$ and $\nabla^E_0$ on $E\to X$, one can take $\e=p_X^*E$
in above and define
\begin{equation}\label{eq 2.1.5}
\nabla^{\e}=\nabla^E_t+dt\wedge\frac{\partial}{\partial t},
\end{equation}
where $\nabla^E_t$ is a smooth curve of connections joining $\nabla^E_0$ and $\nabla^E_1$.
Note that $\CS(\nabla^E_0, \nabla^E_1)$ is independent of the choice of $\nabla^E_t$ (see,
for example, \cite[Proposition 1.1]{SS10} and \cite[Theorem B.5.4]{MM07}). In the case
that $E\to X$ is equipped with Hermitian metrics $g^E_0$ and $g^E_1$, and $\nabla^E_k$ is
unitary with respect to $g^E_k$ for $k\in\set{0, 1}$, one defines $\nabla^{\e}$ in
(\ref{eq 2.1.3}) as in \cite[p.373]{MM07}. We recall its construction here. Denote by
$\ov{E}^*\to X$ the antidual bundle of $E\to X$, and by $f_k:E\to\ov{E}^*$ the canonical
smooth bundle isomorphism associated to $g^E_k$. Then $f_0^{-1}\circ f_1\in\Aut(E)$ is
positive and self-adjoint with respect to $g^E_0$ on each fiber. Denote by $f$ its unique
positive self-adjoint square root. By the proof of \cite[Theorem 8.8 of Chapter 1]{K08} we
have $g^E_1=f^*g^E_0$, and therefore $(f^{-1})^*\nabla^E_1$ is unitary with respect to
$g^E_0$. Thus the smooth path of connections on $E\to X$ defined by
$$\nabla^E_t=(1-t)\nabla^E_0+t(f^{-1})^*\nabla^E_1$$
is unitary with respect to $g^E_0$. Let $f_t=(1-t)\id_E+tf$. Since $g^E_t=f_t^*g^E_0$ is a
smooth path of Hermitian metrics on $E\to X$ joining $g^E_0$ and $g^E_1$, it follows that
$f_t^*\nabla^E_t$ is a smooth path of unitary connections on $E\to X$ with respect to
$g^E_t$ for each $t\in[0, 1]$ joining $\nabla^E_0$ and $\nabla^E_1$. Define a Hermitian
metric $g^{\e}$ on $\e\to\wt{X}$ by $g^{\e}=p_X^*g^E_t$. Then the connection on $\e\to
\wt{X}$ defined by
\begin{equation}\label{eq 2.1.6}
\nabla^{\e}=f_t^*\nabla^E_t+dt\wedge\bigg(\frac{\partial}{\partial t}+\frac{1}{2}
(g^E_t)^{-1}\frac{\partial}{\partial t}g^E_t\bigg)
\end{equation}
is unitary with respect to $g^{\e}$.

Another equivalent definition of the Chern--Simons form is given by
\begin{equation}\label{eq 2.1.7}
\CS(\nabla^E_1, \nabla^E_0)=\int^1_0\tr\bigg(\frac{d\nabla^E_t}{dt}e^{-\frac{1}{2\pi i}
(\nabla^E_t)^2}\bigg)dt.
\end{equation}
The choices of 0 and 1 are immaterial. If $t<T$ are two fixed positive real numbers,
then one can replace 0 by $t$ and 1 by $T$ in above.

It follows from (\ref{eq 2.1.3}) that the Chern--Simons form satisfies the following
properties:
\begin{eqnarray}
\CS(\nabla^E_1, \nabla^E_0)&=&-\CS(\nabla^E_0, \nabla^E_1),\label{eq 2.1.8}\\
\CS(\nabla^E_1, \nabla^E_0)&=&\CS(\nabla^E_1, \nabla^E_2)+\CS(\nabla^E_2,
\nabla^E_0),\label{eq 2.1.9}\\
\CS(\nabla^E_1\oplus\nabla^F_1, \nabla^E_0\oplus\nabla^F_0)&=&\CS(\nabla^E_1,
\nabla^E_0)+\CS(\nabla^F_1, \nabla^F_0).\label{eq 2.1.10}
\end{eqnarray}

Let $E\to X$ be a $\Z_2$-graded complex vector bundle with a superconnection $\aaa$. The
Chern character form of $\aaa$ is defined by
$$\ch(\aaa)=\str(e^{-\frac{1}{2\pi i}\aaa^2})\in\Omega^{\even}_\Q(X; \C),$$
where $\str:\Gamma(X, \Lambda(T^*X)\otimes\C\otimes\End(X))\to\Omega(X; \C)$ is the
extended supertrace \cite[Section 1.5]{BGV}. If $\aaa_0$ and $\aaa_1$ are two
superconnections on $E\to X$, one can define the Chern--Simons form \dis{\CS(\aaa_0,
\aaa_1)\in\frac{\Omega^{\odd}(X; \C)}{\im(d)}} in a similar way as (\ref{eq 2.1.3}) or
equivalently (\ref{eq 2.1.7}). Chern--Simons form of superconnections shares similar
properties to Chern--Simons form of ordinary connections, namely,
\begin{eqnarray}
\CS(\aaa_0, \aaa_1)&=&-\CS(\aaa_1, \aaa_0),\label{eq 2.1.11}\\
\CS(\aaa_0, \aaa_1)&=&\CS(\aaa_0, \aaa_2)+\CS(\aaa_2, \aaa_1),\label{eq 2.1.12}\\
\CS(\aaa_0, \aaa_1)+\CS(\bbb_0, \bbb_1)&=&\CS(\aaa_0\oplus\bbb_0, \aaa_1\oplus\bbb_1).
\label{eq 2.1.13}
\end{eqnarray}
If the superconnections $\aaa_0$ and $\aaa_1$ are $\Z_2$-graded connections $\nabla_0
=\nabla_0^+\oplus\nabla_0^-$ and $\nabla_1=\nabla_1^+\oplus\nabla_1^-$, then one can
easily show that
\begin{equation}\label{eq 2.1.14}
\CS(\nabla_1, \nabla_0)=\CS(\nabla_1^+, \nabla_0^+)-\CS(\nabla_1^-, \nabla_0^-).
\end{equation}
(\ref{eq 2.1.14}) holds for $\Z_2$-graded unitary connections as well.

Now put a $\Z_2$-graded Hermitian metric $g^E$ and a $\Z_2$-graded unitary connection
$\nabla^E$ on $E\to X$. Let $s_E\in\Gamma(X, \End^-(E))$ be an odd self-adjoint section.
The quadruple $(E, g^E, \nabla^E, s_E)$ is said to split \cite[Definition 2.9]{B96} if
\begin{enumerate}
  \item $E=\ker(s_E)\oplus\im(s_E)$,
  \item $\nabla^E$ preserves $\ker(s_E)\to X$ and $\im(s_E)\to X$, and
  \item $s_E$ is a unitary odd section of $\End(\im(s_E))\to X$ preserving
        $\nabla^{\im(s_E)}$.
\end{enumerate}
\begin{exam}\label{exam 1}
An example of split quadruple $(W, g^W, \nabla^W, s_W)$ over $B$ is given by
$$W^+=W^-,\qquad g^{W^+}=g^{W^-},\qquad\nabla^{W^+}=\nabla^{W^-},\qquad s_W=
\begin{pmatrix} 0 & \id \\ \id & 0 \end{pmatrix}.$$
In this case $\ker(s_W)\to B$ is the zero $\Z_2$-graded complex vector bundle, so the
other conditions of $(W, g^W, \nabla^W, s_W)$ being split are immediately satisfied.
\end{exam}

Example \ref{exam 1} will play an important role in proving Proposition \ref{prop 3.1.1}.

Let $E\to X$ be a real vector bundle with a Euclidean metric $g^E$ and a Euclidean
connection $\nabla^E$. The $\wh{A}$-genus form of $\nabla^E$ is defined to be
$$\wh{A}(\nabla^E)=\sqrt{\det\bigg(\frac{-\frac{1}{4\pi i}R^E}{\sinh(-\frac{1}{4\pi i}
R^E)}\bigg)}\in\Omega^{4\bullet}_\Q(X),$$
where $R^E$ is the curvature of $\nabla^E$. If $\nabla^E_1$ and $\nabla^E_0$ are two
Euclidean connections on $E\to X$, one can define a transgression form \dis{\wt{\wh{A}}
(\nabla^E_1, \nabla^E_0)\in\frac{\Omega^{4\bullet-1}(X)}{\im(d)}} between $\wh{A}
(\nabla^E_0)$ and $\wh{A}(\nabla^E_1)$ by
\begin{equation}\label{eq 2.1.16}
\wt{\wh{A}}(\nabla^E_0, \nabla^E_1)=-\int_{\wt{X}/X}\wh{A}(\nabla^{\e}),
\end{equation}
where $\nabla^{\e}$ is the connection on $\e\to\wt{X}$ defined in a way similar to
(\ref{eq 2.1.6}). Similarly we have
$$d\wt{\wh{A}}(\nabla^E_0, \nabla^E_1)=\wh{A}(\nabla^E_1)-\wh{A}(\nabla^E_0).$$

\subsection{Cheeger--Chern--Simons class}\label{s 2.2}

In this subsection we review the definition of the Cheeger--Chern--Simons class of complex
flat vector bundles. We refer to \cite{BL95, B05, MZ08} for the details.

Let $F\to X$ be a complex flat vector bundle with flat connection $\nabla^F$. The
Cheeger--Chern--Simons class $\CCS(F, \nabla^F)\in H^{\odd}(X; \C/\Q)$ of $(F, \nabla^F)$
\cite[\S1(g)]{BL95} (see also \cite[Definition 2.11]{MZ08}) is defined as follows. Denote
by $\C^N\to X$ the trivial complex vector bundle of rank $N$. Since
$\ch([F]-[\C^{\rk(F)}])=0\in H^{\even}(X; \Q)$, by \cite[p.89]{APS76} there exists
$k\in\N$ such that $kF\cong k\C^{\rk(F)}$ as smooth complex vector bundles.
\begin{remark}\label{remark 1}
This fact can be proved as follows. Let $E_1\to X$ and $E_2\to X$ be complex vector
bundles of the same rank $\ell$. Recall from \cite[Theorem 1.5 of Chapter 9]{H94} that if
$2\ell\geq\dim(X)$ and $E_1\oplus\C^m\cong E_2\oplus\C^m$ for some $m\in\N$, then
$E_1\cong E_2$.

Now suppose $E_1\to X$ and $E_2\to X$ have the same rank $\ell$ and satisfy $\ch([E_1]-
[E_2])=0$. Since $\ch:K^0(X)\otimes\Q\to H^{\even}(X; \Q)$ is a ring isomorphism, there
exists $k\in\N$ such that $k[E_1]=k[E_2]$. Since it also holds for any integer multiple
of $k$, we can take a sufficiently large integer multiple of $k$, still denoted by $k$,
so that $2k\ell\geq\dim(X)$. Since $kE_1\oplus\C^m\cong kE_2\oplus\C^m$ for some $m\in\N$,
it follows from \cite[Theorem 1.5 of Chapter 9]{H94} that $kE_1\cong kE_2$.
\end{remark}
Let $\nabla^{kF}_0$ be a trivial connection on $kF\to X$, which can be determined by
choosing a global frame for $kF\to X$. One can check that the odd form \dis{\frac{1}{k}
\CS(\nabla^{kF}_0, k\nabla^F)} is closed. The Cheeger--Chern--Simons class of
$(F, \nabla^F)$ is defined to be
\begin{equation}\label{eq 2.2.1}
\CCS(F, \nabla^F):=\bigg[\frac{1}{k}\CS(\nabla^{kF}_0, k\nabla^F)\bigg]\mod\Q.
\end{equation}
Note that $\CCS(F, \nabla^F)$ is independent of the choices of $k\in\N$ such that
$kF\cong k\C^{\rk(F)}$ \cite[Theorem B.5.4]{MM07} and $\nabla^{kF}_0$ (by proceeding
as in \cite[Lemma 1]{L94}).

Put a Hermitian metric $g^F$ on $F\to X$. As in \cite[Definition 4.1]{BZ92} define
\begin{equation}\label{eq 2.2.2}
\omega(F, g^F):=(g^F)^{-1}(\nabla^Fg^F)\in\Omega^1(X, \End(F)).
\end{equation}
By \cite[Definition 4.2, Proposition 4.3]{BZ92}, the connection $\nabla^{F, u}$ on
$F\to X$ defined by
\begin{equation}\label{eq 2.2.3}
\nabla^{F, u}=\nabla^F+\frac{1}{2}\omega(F, g^F)
\end{equation}
is unitary with respect to $g^F$ and has curvature
\begin{equation}\label{eq 2.2.4}
(\nabla^{F, u})^2=-\frac{1}{4}\omega(F, g^F)^2.
\end{equation}
By \cite[(2.33), (2.37)]{B05}, $\tr(\omega(F, g^F)^{2k})=0$ for any $k\in\N$, and
therefore
\begin{equation}\label{eq 2.2.5}
\ch(\nabla^{F, u})=\rk(F).
\end{equation}
By \cite[(2.44)]{MZ08} the real part of $\CCS(F, \nabla^F)$ is given by
\begin{equation}\label{eq 2.2.6}
\re(\CCS(F, \nabla^F))=\bigg[\frac{1}{k}\CS(\nabla^{kF}_0, k\nabla^{F, u})\bigg]\mod\Q
\in H^{\odd}(X; \R/\Q).
\end{equation}
\begin{remark}\label{remark 2}
The Cheeger--Chern--Simons class $\CCS(F, \nabla^F)$ measures the deviation of $\nabla^F$
(or more precisely its $k$-fold direct sum) from being a trivial connection. If $\nabla^F$
(or $k\nabla^F$) is indeed a trivial connection, then $k\nabla^F$ differs from
$\nabla^{kF}_0$ by a gauge transformation specified by a smooth map $g:X\to\GL(k\rk(F);
\C)$. Thus $\CS(\nabla^{kF}_0, k\nabla^F)=\ch^{\odd}(g)\in\Omega^{\odd}_\Q(X; \C)$. Since
$\CCS(F, \nabla^F)$ is the mod $\Q$ reduction of the de Rham class of a rational multiple
of $\CS(\nabla^{kF}_0, k\nabla^F)$, it follows that $\CCS(F, \nabla^F)=0\in H^{\odd}
(X; \C/\Q)$.
\end{remark}

Let $(E, v)$ be a $\Z$-graded cochain complex of complex vector bundles over $X$, i.e.
\begin{equation}\label{eq 2.2.7}
\begin{CD}
0 @>>> E^0 @>v>> E^1 @>v>> \cdots @>v>> E^m @>>> 0
\end{CD}
\end{equation}
with $v\circ v=0$. Let $\nabla^k$ be a connection on $E^k\to X$ for each $0\leq k
\leq m$. Define
$$E=\bigoplus_{k=0}^mE^k\qquad\textrm{ and }\qquad\nabla^E=\bigoplus_{k=0}^m\nabla^k.$$
Note that $\nabla^E$ is a $\Z$-graded connection on $E\to X$. The triple $(E, v,
\nabla^E)$ is called a flat cochain complex if $\nabla^E$ is a flat connection on
$E\to X$ and $[\nabla^E, v]=0$. Define
\begin{equation}\label{eq 2.2.8}
\begin{split}
E^+&=\bigoplus_kE^{2k},\qquad E^-=\bigoplus_kE^{2k+1},\\
\nabla^+&=\bigoplus_k\nabla^{2k},\qquad \nabla^-=\bigoplus_k\nabla^{2k+1}.
\end{split}
\end{equation}
Then $E\to X$, where $E=E^+\oplus E^-$, is a $\Z_2$-graded complex vector bundle with a
$\Z_2$-graded connection $\nabla^E=\nabla^+\oplus\nabla^-$.

For any $\Z$-graded cochain complex $(E, v)$ with two $\Z$-graded connections
$\nabla^E_1$ and $\nabla^E_0$ define
\begin{displaymath}
\begin{split}
\ch(E, \nabla^E_j)&:=\sum_{k=0}^m(-1)^k\ch(E^k, \nabla^k_j) \textrm{ where } j=0, 1,\\
\CS(\nabla^E_1, \nabla^E_0)&:=\sum_{k=0}^m(-1)^k\CS(\nabla^k_1, \nabla^k_0).
\end{split}
\end{displaymath}
The Cheeger--Chern--Simons class of a $\Z$-graded flat cochain complex $(E, v, \nabla^E)$
is defined to be
\begin{equation}\label{eq 2.2.9}
\CCS(E, \nabla^E)=\sum_{k=0}^m(-1)^k\CCS(E^k, \nabla^k)\in H^{\odd}(X; \C/\Q).
\end{equation}

Let $F\to X$ be a $\Z_2$-graded complex flat vector bundle with $\Z_2$-graded flat
connection $\nabla^F=\nabla^+\oplus\nabla^-$ of virtual rank zero. Since $\ch(\nabla^+)
-\ch(\nabla^-)=\rk(F^+)-\rk(F^-)=0$, there exist $k\in\N$ and a smooth bundle isomorphism
$j:kF^+\to kF^-$. The Cheeger--Chern--Simons class of $(F, \nabla^F)$ is given by
\begin{equation}\label{eq 2.2.10}
\CCS(F, \nabla^F)=\bigg[\frac{1}{k}\CS(j^*k\nabla^-, k\nabla^+)\bigg]\mod\Q.
\end{equation}
Note that $\CCS(F, \nabla^F)$ is independent of the choices of $k$
\cite[Theorem B.5.4]{MM07} and $j$.
\begin{remark}\label{remark 3}
To prove that $\CCS(F, \nabla^F)$ is independent of the choice of $j$, suppose $j_1:kF^+\to
kF^-$ is another smooth bundle isomorphism. By (\ref{eq 2.1.8}) and (\ref{eq 2.1.9})
we have
\begin{displaymath}
\begin{split}
\CS(j^*k\nabla^-, k\nabla^+)-\CS(j_1^*k\nabla^-, k\nabla^+)=&\CS(j^*k\nabla^-, k\nabla^+)
+\CS(k\nabla^+, j_1^*k\nabla^-)\\
=&\CS(j^*k\nabla^-, j_1^*k\nabla^-)\\
=&\CS(k\nabla^-, (j^{-1})^*j_1^*k\nabla^-)\\
=&\CS(k\nabla^-, (j_1\circ j^{-1})^*k\nabla^-),
\end{split}
\end{displaymath}
where the third equality follows from the fact that $j$ covers the identity map $\id_X$.
Since $j_1\circ j^{-1}\in\Aut(kF^-)$, it follows that $(kF^-, j_1\circ j^{-1})$ defines
an element in $K^{-1}(X)$ \cite[p.71-p.73]{K08}. It is well known that for any $[(E, f)]
\in K^{-1}(X)$, a differential form representative of the odd Chern character
$\ch^{\odd}([(E, f)])\in H^{\odd}(X; \Q)$ is given by $\CS(\nabla^E, f^*\nabla^E)$, where
$\nabla^E$ is any connection on $E\to X$ (see \cite[p.955]{FL10} for its Hermitian analog).
Thus
$$\CS(j^*k\nabla^-, k\nabla^+)-\CS(j_1^*k\nabla^-, k\nabla^+)=\CS(k\nabla^-,
(j_1\circ j^{-1})^*k\nabla^-)\in\Omega^{\odd}_\Q(X; \C).$$
\end{remark}

To obtain the formula of $\re(\CCS(F, \nabla^F))$, put a $\Z_2$-graded Hermitian metric
$g^F=g^+\oplus g^-$ on $F\to X$, and define unitary connections $\nabla^{\pm, u}$ on
$F^\pm\to X$ with respect to $g^\pm$ by (\ref{eq 2.2.3}). Since $kg^+$ and $j^*kg^-$ are
Hermitian metrics on $kF^+\to X$, by the proof of \cite[Theorem 8.8 of Chapter 1]{K08}
there exists $f\in\Aut(kF^+)$ such that
\begin{equation}\label{eq 2.2.11}
kg^+=f^*j^*kg^-=(j\circ f)^*kg^-.
\end{equation}
Write $\wt{j}$ for $j\circ f$. Note that $\wt{j}^*k\nabla^{-, u}$ is unitary with respect
to $kg^+$. Since $\wt{j}$ covers the identity map $\id_X$, it follows from
(\ref{eq 2.1.9}) and (\ref{eq 2.1.10}) that
\begin{equation}\label{eq 2.2.12}
\begin{split}
&\CS(\wt{j}^*k\nabla^-, k\nabla^+)\\
=&\CS(\wt{j}^*k\nabla^-, \wt{j}^*k\nabla^{-, u})+\CS(\wt{j}^*k\nabla^{-, u},
k\nabla^{+, u})+\CS(k\nabla^{+, u}, k\nabla^+)\\
=&\CS(k\nabla^-, k\nabla^{-, u})+\CS(\wt{j}^*k\nabla^{-, u}, k\nabla^{+, u})+\CS
(k\nabla^{+, u}, k\nabla^+).
\end{split}
\end{equation}
Since $i\im(\CCS(F^\pm, \nabla^\pm))=\CS(\nabla^{\pm, u}, \nabla^\pm)$ by
\cite[(2.43)]{MZ08}, it follows from (\ref{eq 2.2.10}) and (\ref{eq 2.2.12}) that
\begin{equation}\label{eq 2.2.13}
\re(\CCS(F, \nabla^F))=\bigg[\frac{1}{k}\CS(\wt{j}^*k\nabla^{-, u}, k\nabla^{+, u})\bigg]
\mod\Q.
\end{equation}

It is necessary to define $\re(\CCS(F, \nabla^F))$ in terms of $\wt{j}^*k\nabla^{-, u}$
instead of $j^*k\nabla^{-, u}$. Since the connection $j^*k\nabla^{-, u}$ is not unitary
with respect to $kg^+$ in general, if it was used to define $\re(\CCS(F, \nabla^F))$ then
the Chern--Simons form on the right-hand side of (\ref{eq 2.2.13}) would not be
real-valued.

Since $\CCS(F, \nabla^F)$ is independent of the choices of $k\in\N$ and $j:kF^+\to kF^-$,
and $f$ is uniquely determined by $(k, j, g^F)$, $\re(\CCS(F, \nabla^F))$ is independent
of the choices made as well.

\subsection{Local index theory for twisted de Rham operator}\label{s 2.3}

In this subsection we recall the setup and the statement of the local FIT for twisted
de Rham operator \cite[\S3]{B05} (see also \cite{BL95, MZ08}).

Let $\pi:X\to B$ be a submersion with closed fibers $Z$ of dimension $n$. Denote by
$T^VX\to X$ the vertical tangent bundle. Let $T^HX\to X$ be a horizontal distribution
for $\pi:X\to B$, i.e. $TX=T^VX\oplus T^HX$. Denote by $P^{T^VX}:TX\to T^VX$ the
projection map. Put a metric $g^{T^VX}$ on $T^VX\to X$. Given a Riemannian metric
$g^{TB}$ on $TB\to B$, define a metric on $TX\to X$ by
$$g^{TX}:=g^{T^VX}\oplus\pi^*g^{TB}.$$
If $\nabla^{TX}$ is the corresponding Levi-Civita connection on $TX\to X$, then
$\nabla^{T^VX}:=P^{T^VX}\nabla^{TX}$ is an Euclidean connection on $T^VX\to X$ with
respect to $g^{T^VX}$.

The exterior bundle $\Lambda(T^VX)^*\to X$ is a Clifford module with Clifford
multiplication $c(Y)=\varepsilon(Y)-i(Y)$, where $Y\in\Gamma(X, T^VX)$, $\varepsilon$ is
the exterior multiplication and $i$ is the interior multiplication. Here $T^VX\to X$ is
identified with $(T^VX)^*\to X$ via $g^{T^VX}$. Denote by $\nabla^{\Lambda(T^VX)^*}$ the
extension of $\nabla^{T^VX}$ on $T^VX\to X$ to $\Lambda(T^VX)^*\to X$. Set
$$\wh{c}(Y)=\varepsilon(Y)+i(Y).$$

Let $F\to X$ be a complex flat vector bundle with flat connection $\nabla^F$. Put a
Hermitian metric $g^F$ on $F\to X$. Define a twisted Dirac operator $\DD^{\Lambda
\otimes F}:\Gamma(X, \Lambda(T^VX)^*\otimes F)\to\Gamma(X, \Lambda(T^VX)^*\otimes F)$
by
\begin{equation}\label{eq 2.3.1}
\DD^{\Lambda\otimes F}=\sum_{k=1}^nc(e_k)\nabla_{e_k}^{\Lambda(T^VX)^*\otimes F, u},
\end{equation}
where $\nabla^{\Lambda(T^VX)^*\otimes F, u}$ is the tensor product of
$\nabla^{\Lambda(T^VX)^*}$ and $\nabla^{F, u}$, and $\set{e_k}$ is a local orthonormal
frame for $T^VX\to X$. Define an infinite rank $\Z$-graded complex vector bundle
$\pi^\Lambda_*F\to B$ whose fiber over $b\in B$ is
\begin{equation}\label{eq 2.3.2}
(\pi^\Lambda_*F)_b:=\Gamma(Z_b, (\Lambda(T^VX)^*\otimes F)|_{Z_b}).
\end{equation}
By \cite[(3.6)]{BL95} we have
$$\Omega(X, F)\cong\Omega(B, \pi^\Lambda_*F).$$
Denote by $\ast$ the fiberwise Hodge star operator associated to $g^{T^VX}$, and extend
it from $\Gamma(X, \Lambda(T^VX)^*)$ to $\Gamma(X, \Lambda(T^VX)^*\otimes F)\cong\Gamma
(B, \pi^\Lambda_*F)$. Define an $L^2$-metric on $\pi^\Lambda_*F\to B$ by
$$g^{\pi^\Lambda_*F}(s_1, s_2)(b)=\int_{Z_b}g^F(s_1(b)\wedge\ast s_2(b)).$$

Let $U\in\Gamma(B, TB)$ and denote by $U^H\in\Gamma(X, T^HX)$ its lift. Define a
connection $\nabla^{\pi^\Lambda_*F}$ on $\pi_*^\Lambda F\to B$ by
$$\nabla^{\pi^\Lambda_*F}_Us=\nabla^{\Lambda(T^VX)^*\otimes F, u}_{U^H}s,$$
where $s\in\Gamma(B, \pi^\Lambda_*F)$. Note that $\nabla^{\pi^\Lambda_*F}$ preserves
the $\Z$-grading of $\pi^\Lambda_*F\to B$. Denote by $\nabla^{TB}$ the Levi-Civita
connection associated to $g^{TB}$, and define $^0\nabla^{TX}=\pi^*\nabla^{TB}\oplus
\nabla^{T^VX}$. Define $S:=\nabla^{TX}-^0\nabla^{TX}\in\Omega^1(X, \End(TX))$. By
\cite[Theorem 1.9]{B86}, the $(3, 0)$ tensor $g^{TX}(S(\cdot)\cdot, \cdot)$ only depends
on $(T^HX, g^{T^VX})$. Define a horizontal one-form $k$ on $X$ by
\begin{equation}\label{eq 2.3.3}
k(U^H)=-\sum_{k=1}^ng^{TX}(S(e_k)e_k, U^H).
\end{equation}
The connection $\nabla^{\pi^\Lambda_*F, u}$ on $\pi^\Lambda_*F\to B$ defined by
\begin{equation}\label{eq 2.3.4}
\nabla^{\pi^\Lambda_*F, u}:=\nabla^{\pi^\Lambda_*F}+\frac{1}{2}k
\end{equation}
is unitary with respect to $g^{\pi^\Lambda_*F}$ \cite[Proposition 1.4]{BF86a}.

Denote by $d^Z$ the fiberwise de Rham operator coupled with $\nabla^F$ acting on
$\pi^\Lambda_*F\to B$. The connection $\wt{\nabla}^{\pi^\Lambda_*F}$
\cite[Definition 3.2]{BL95} on $\pi^\Lambda_*F\to B$ defined by
$$\wt{\nabla}^{\pi^\Lambda_*F}_Us:=\LLL_{U^H}s,$$
where $U\in\Gamma(B, TB)$, $U^H\in\Gamma(X, T^HX)$ is its lift, $\LLL_{U^H}:\Gamma(X,
\Lambda(T^VX)^*\otimes F)\to\Gamma(X, \Lambda(T^VX)^*\otimes F)$ is the Lie derivative
and $s\in\Gamma(B, \pi^\Lambda_*F)$, is $\Z$-graded.

The exterior differential $d^X:\Omega(X, F)\to\Omega(X, F)$ coupled with $\nabla^F$ can
be regarded as a flat superconnection on $\pi^\Lambda_*F\to B$ of total degree 1 whose
decomposition \cite[Proposition 3.4]{BL95} is given by
$$d^X=d^Z+\wt{\nabla}^{\pi^\Lambda_*F}+i_T,$$
where $T$ is the curvature 2-form of the fiber bundle $X\to B$, and $i_T$ is given in
\cite[Definition 3.3]{BL95}.

Consider $d^Z$ as an element in $\Gamma(B, \ho((\pi^\Lambda_*F)^\bullet,
(\pi^\Lambda_*F)^{\bullet+1}))$. For each $b\in B$,
\cdd{0 @>>> (\pi^\Lambda_*F)_b^0 @>d^{Z_b}>> (\pi^\Lambda_*F)_b^1 @>d^{Z_b}>>
\cdots @>d^{Z_b}>> (\pi^\Lambda_*F)_b^n @>>> 0}
is a cochain complex. Denote by $H^k(Z_b, F|_{Z_b})$ the associated $k$-th cohomology
group and define \dis{H(Z_b, F|_{Z_b}):=\bigoplus_{k=0}^nH^k(Z_b, F|_{Z_b})}. Define a
$\Z$-graded complex vector bundle $H(Z, F|_Z)\to B$ whose fiber over $b\in B$ is given
by $H(Z, F|_Z)_b:=H(Z_b, F|_{Z_b})$. Denote by $\psi:\ker(d^Z)\to H(Z, F|_Z)$ the
quotient map. For $s\in\Gamma(B, H^k(Z, F|_Z))$, let $e\in\Gamma(B, (\pi^\Lambda_*F)^k
\cap\ker(d^Z))$ be such that $\psi(e)=s$. The connection $\nabla^{H(Z, F|_Z)}$ on
$H(Z, F|_Z)\to B$ \cite[Definition 2.4]{BL95} defined by
$$\nabla^{H(Z, F|_Z)}_Us=\psi(\wt{\nabla}^{\pi^\Lambda_*F}_Ue),$$
where $U\in\Gamma(B, TB)$, is a well defined $\Z$-graded flat connection
\cite[Definition 2.4, Proposition 2.5]{BL95}.

Denote by $d^{Z\ast}$ the formal adjoint of $d^Z$ with respect to $g^{\pi_*^{\Lambda}F}$.
As in \cite[Definition 3.8]{BL95} define
\begin{displaymath}
\begin{split}
\DD^{Z, \dr}&=d^Z+d^{Z\ast},\\
\wt{\nabla}^{\pi^\Lambda_*F, u}&=\frac{1}{2}(\wt{\nabla}^{\pi^\Lambda_*F}+
(\wt{\nabla}^{\pi^\Lambda_*F})^*),\\
\omega(\pi^\Lambda_*F, g^{\pi^\Lambda_*F})&=(\wt{\nabla}^{\pi^\Lambda_*F})^*-
\wt{\nabla}^{\pi^\Lambda_*F},
\end{split}
\end{displaymath}
where $(\wt{\nabla}^{\pi^\Lambda_*F})^*$ is the adjoint of $\wt{\nabla}^{\pi^\Lambda_*F}$
with respect to $g^{\pi^\Lambda_*F}$ \cite[Definition 1.6]{BL95}. By Hodge theory we have
$$H(Z_b, F|_{Z_b})\cong\ker(\DD^{Z_b, \dr}).$$
Define a $\Z$-graded complex vector bundle $\ker(\DD^{Z, \dr})\to B$ whose fiber over
$b\in B$ is given by $\ker(\DD^{Z, \dr})_b:=\ker(\DD^{Z_b, \dr})$. Then $\ker(\DD^{Z, \dr})
\to B$ is a finite rank subbundle of $\pi^\Lambda_*F\to B$ and
\begin{equation}\label{eq 2.3.5}
H(Z, F|_Z)\cong\ker(\DD^{Z, \dr})
\end{equation}
as $\Z$-graded complex vector bundles. Note that $\ker(\DD^{Z, \dr})\to B$ inherits a
$\Z$-graded Hermitian metric from $g^{\pi^\Lambda_*F}$, which will be denoted by
$g^{\ker(\DD^{Z, \dr})}$. Denote by $g^{H(Z, F|_Z)}$ the $\Z$-graded Hermitian metric on
$H(Z, F|_Z)\to B$ obtained by pulling back $g^{\ker(\DD^{Z, \dr})}$ via the isomorphism
(\ref{eq 2.3.5}). Denote by $P^{\ker(\DD^{Z, \dr})}:\pi^\Lambda_*F\to\ker(\DD^{Z, \dr})$
the orthogonal projection onto $\ker(\DD^{Z, \dr})$. Then
$P^{\ker(\DD^{Z, \dr})}\wt{\nabla}^{\pi^\Lambda_*F}$ is a connection on $\ker(\DD^{Z, \dr})
\to B$, which can be considered as a connection on $H(Z, F|_Z)\to B$ via the isomorphism
(\ref{eq 2.3.5}). By \cite[Proposition 3.14]{BL95} we have
\begin{displaymath}
\begin{split}
\nabla^{H(Z, F|_Z)}&=P^{\ker(\DD^{Z, \dr})}\wt{\nabla}^{\pi^\Lambda_*F},\\
\omega(H(Z, F|_Z), g^{H(Z, F|_Z)})&=P^{\ker(\DD^{Z, \dr})}\omega(\pi^\Lambda_*F,
g^{\pi^\Lambda_*F})P^{\ker(\DD^{Z, \dr})}.
\end{split}
\end{displaymath}
Define \dis{\nabla^{H(Z, F|_Z), u}:=\nabla^{H(Z, F|_Z)}+\frac{1}{2}\omega(H(Z, F|_Z),
g^{H(Z, F|_Z)})}. Consequently we have
$$\nabla^{H(Z, F|_Z), u}=P^{\ker(\DD^{Z, \dr})}\wt{\nabla}^{\pi^\Lambda_*F, u}.$$
By \cite[(3.38)]{BL95} we have
\begin{equation}\label{eq 2.3.6}
\wt{\nabla}^{\pi^\Lambda_*F, u}=\nabla^{\pi^\Lambda_*F, u},
\end{equation}
where the left-hand side of (\ref{eq 2.3.6}) is defined by (\ref{eq 2.2.3}), and the
right-hand side of (\ref{eq 2.3.6}) is defined by (\ref{eq 2.3.4}). Therefore we have
\begin{equation}\label{eq 2.3.7}
\nabla^{H(Z, F|_Z), u}=P^{\ker(\DD^{Z, \dr})}\nabla^{\pi^\Lambda_*F, u}.
\end{equation}

By \cite[Proposition 4.12]{BZ92} we have
\begin{equation}\label{eq 2.3.8}
\DD^{Z, \dr}=\DD^{\Lambda\otimes F}+V,
\end{equation}
where
\begin{equation}\label{eq 2.3.9}
V=-\frac{1}{2}\sum_{k=1}^n\wh{c}(e_k)\omega(F, g^F)(e_k).
\end{equation}
Note that $V$ is an odd self-adjoint matrix-valued operator which anti-commutes with the
$c(X)$'s.

Define the Bismut superconnection $\bbb^{\dr}$ on $\pi^\Lambda_*F\to B$ \cite[(3.49)]{B05}
associated to $\DD^{Z, \dr}$ by
\begin{equation}\label{eq 2.3.10}
\bbb^{\dr}:=\DD^{Z, \dr}+\nabla^{\pi^\Lambda_*F, u}-\frac{c(T)}{4}.
\end{equation}
The rescaled Bismut superconnection $\bbb^{\dr}_t$ is given by
$$\bbb^{\dr}_t:=\sqrt{t}\DD^{Z, \dr}+\nabla^{\pi^\Lambda_*F, u}-\frac{c(T)}{4\sqrt{t}}.$$

One might consider the ``unperturbed" Bismut superconnection $\bbb$ associated to
$\DD^{\Lambda\otimes F}$, i.e. \dis{\bbb=\DD^{\Lambda\otimes F}+
\nabla^{\pi_*^\Lambda F, u}-\frac{c(T)}{4}}, instead of $\bbb^{\dr}$ defined by
(\ref{eq 2.3.10}) whose degree zero term $\bbb^{\dr}_{[0]}$ is perturbed by $V$. The reason
of considering $\bbb^{\dr}$ is because our target is the cohomology bundle $H(Z, F|_Z)
\to B$, which is isomorphic to $\ker(\DD^{\Lambda\otimes F}+V)\to B$. Here the notion of
superconnection is generalized in the sense of \cite[p.286]{BGV}. By
\cite[(3.40), (3.45), (3.46), (3.50)]{BL95} and the proof of \cite[Theorem 3.15]{BL95},
techniques of local index theory can still be applied to $\bbb^{\dr}$. See also the last
paragraph of \cite[p.33]{B05}.

The Bismut--Cheeger eta form $\wt{\eta}^{\dr}$ associated to $\bbb^{\dr}$ is defined by
$$\wt{\eta}^{\dr}:=\int^\infty_0\str\bigg(\frac{d\bbb^{\dr}_t}{dt}e^{-\frac{1}{2\pi i}
(\bbb^{\dr}_t)^2}\bigg)dt.$$

For $\dim(Z)$ even, the local FIT for the twisted de Rham operator $\DD^{Z, \dr}$
\cite[Theorem 3.15]{BL95} (see also \cite[(3.50)]{B05}) is given by
\begin{equation}\label{eq 2.3.11}
d\wt{\eta}^{\dr}=\int_{X/B}e(\nabla^{T^VX})\wedge\ch(\nabla^{F, u})-\ch
(\nabla^{H(Z, F|_Z), u}),
\end{equation}
where $e(\nabla^{T^VX})\in\Omega^n_\Z(X; o(T^VX))$ is the $o(T^VX)$-valued Euler form.
Here $o(T^VX)\to X$ is the orientation bundle of $T^VX\to X$. It is a flat real line
bundle, which is trivial if and only if $T^VX\to X$ is oriented. Note that
$e(\nabla^{T^VX})=0$ if $n$ is odd. \emph{A priori} (\ref{eq 2.3.11}) holds under the
assumptions that the fibers $Z$ are oriented and spin. However, as noted in the proof
of \cite[Theorem 3.15]{BL95}, the computations involved are local, so (\ref{eq 2.3.11})
is still valid if the fibers are not spin and not even orientable.

By (\ref{eq 2.2.5}) we have
\begin{displaymath}
\begin{split}
&\int_{X/B}e(\nabla^{T^VX})\wedge\ch(\nabla^{F, u})-\ch(\nabla^{H(Z, F|_Z), u})\\
=~&\rk(F)\chi(Z)-\rk(H(Z, F|_Z))=0,
\end{split}
\end{displaymath}
where $\chi(Z):B\to\Z$ is the Euler characteristic function. Thus $\wt{\eta}^{\dr}$ is
a closed form. A stronger result by Bismut \cite[Theorem 3.7]{B05} (see also
\cite[(3.74)]{BL95}) states that
\begin{equation}\label{eq 2.3.12}
\wt{\eta}^{\dr}=0.
\end{equation}

\subsection{Local index theory for twisted spin Dirac operator}\label{s 2.4}

In this subsection we review the setup and the statement of the local FIT for twisted spin
Dirac operator under and without the kernel bundle assumption. We refer to \cite[\S5]{L94}
and \cite[\S7]{FL10} for the details.

Let $\pi:X\to B$ be a submersion with closed, oriented and spin fibers $Z$ of even
dimension $n$. Endow $T^VX\to X$ the geometric data as in Section \ref{s 2.3}. Denote by
$S(T^VX)\to X$ the $\Z_2$-graded spinor bundle. The connection $\nabla^{T^VX}$ on $T^VX
\to X$ lifts uniquely to $S(T^VX)\to X$ and preserves its grading.

Let $E\to X$ be a complex vector bundle with a Hermitian metric $g^E$ and a unitary
connection $\nabla^E$. Then $S(T^VX)\otimes E\to X$ is a $\Z_2$-graded complex vector
bundle equipped with the Hermitian metric $g^{T^VX}\otimes g^E$ and the unitary connection
$\nabla^{S(T^VX)\otimes E}$, which is defined to be the tensor product of
$\nabla^{S(T^VX)}$ and $\nabla^E$. The twisted spin Dirac operator $\DD^{S\otimes E}$
acting on $\Gamma(X, S(T^VX)\otimes E)$ is defined to be
\begin{equation}\label{eq 2.4.0}
\DD^{S\otimes E}=\sum_{k=1}^nc(e_k)\nabla^{S(T^VX)\otimes E}_{e_k},
\end{equation}
where $c$ is the Clifford multiplication, and $\set{e_k}$ is a local orthonormal frame
for $T^VX\to X$.

Define an infinite rank $\Z_2$-graded complex vector bundle $\pi^{\spin}_*E\to B$ whose
fibers over $b\in B$ is given by
\begin{equation}\label{eq 2.4.1}
(\pi^{\spin}_*E)_b:=\Gamma(Z_b, (S(T^VX)\otimes E)|_{Z_b}).
\end{equation}
Denote by $g^{\pi^{\spin}_*E}$ the $L^2$-metric on $\pi^{\spin}_*E\to B$ \cite[(1.11)]{B05}.
Define a connection $\nabla^{\pi^{\spin}_*E}$ on $\pi^{\spin}_*E\to B$ by
$$\nabla^{\pi^{\spin}_*E}_Us:=\nabla^{S(T^VX)\otimes E}_{U^H}s,$$
where $s\in\Gamma(B, \pi^{\spin}_*E)$, $U\in\Gamma(B, TB)$ and $U^H\in\Gamma(X, T^HX)$
is a lift of $U$. The connection $\nabla^{\pi^{\spin}_*E, u}$ on $\pi^{\spin}_*E\to B$
defined by
\begin{equation}\label{eq 2.4.2}
\nabla^{\pi^{\spin}_*E, u}:=\nabla^{\pi^{\spin}_*E}+\frac{1}{2}k,
\end{equation}
where $k$ is given by (\ref{eq 2.3.3}), is $\Z_2$-graded and unitary with respect to
$g^{\pi^{\spin}_*E}$ \cite[Proposition 1.4]{BF86a}.

\begin{ass}\label{ass 1}
The family of complex vector spaces $\ker(\DD^{S\otimes E}_b)$, $b\in B$, has locally
constant dimension.
\end{ass}

If Assumption \ref{ass 1} is satisfied, then the family of complex vector spaces
$\ker(\DD^{S\otimes E}_b)$ form a finite rank $\Z_2$-graded subbundle of $\pi^{\spin}_*E
\to B$, denoted by $\ker(\DD^{S\otimes E})\to B$. In this case the family of complex
vector spaces $\im(\DD^{S\otimes E}_b)$ form an infinite rank $\Z_2$-graded subbundle of
$\pi^{\spin}_*E\to B$ as well, denoted by $\im(\DD^{S\otimes E})\to B$. Their
$\Z_2$-gradings are given by
$$\ker(\DD^{S\otimes E})^\pm=\ker(\DD^{S\otimes E}_\pm),\qquad
\im(\DD^{S\otimes E})^\pm=\im(\DD^{S\otimes E}_\mp).$$
Note that the direct sum decompositions
\begin{displaymath}
\begin{split}
(\pi^{\spin}_*E)^+&=\im(\DD^{S\otimes E})^+\oplus\ker(\DD^{S\otimes E})^+,\\
(\pi^{\spin}_*E)^-&=\im(\DD^{S\otimes E})^-\oplus\ker(\DD^{S\otimes E})^-
\end{split}
\end{displaymath}
are orthogonal. The $K$-theoretic analytic index of $[E]\in K(X)$ can be defined as
$\ind^{\an}([E])=[\ker(\DD^{S\otimes E})^+]-[\ker(\DD^{S\otimes E})^-]$.

The Bismut superconnection $\bbb^E$ on $\pi^{\spin}_*E\to B$ is defined to be
\begin{equation}
\bbb^E=\DD^{S\otimes E}+\nabla^{\pi^{\spin}_*E, u}-\frac{c(T)}{4},
\end{equation}
where $T$ is the curvature 2-form of the fiber bundle $\pi:X\to B$. The rescaled
Bismut superconnection $\bbb^E_t$ is given by
$$\bbb^E_t=\sqrt{t}\DD^{S\otimes E}+\nabla^{\pi^{\spin}_*E, u}-\frac{c(T)}{4\sqrt{t}}.$$

Denote by $P^{\ker(\DD^{S\otimes E})}:\pi^{\spin}_*E\to\ker(\DD^{S\otimes E})$ the
orthogonal projection. Then $g^{\ker(\DD^{S\otimes E})}:=P^{\ker(\DD^{S\otimes E})}
g^{\pi^{\spin}_*E}$ is a Hermitian metric on $\ker(\DD^{S\otimes E})\to B$. Moreover,
$\nabla^{\ker(\DD^{S\otimes E})}:=P^{\ker(\DD^{S\otimes E})}\nabla^{\pi^{\spin}_*E, u}$
is a $\Z_2$-graded unitary connection on $\ker(\DD^{S\otimes E})\to B$. Then
\begin{eqnarray}
\lim_{t\to 0}\ch(\bbb^E_t)&=&\int_{X/B}\wh{A}(\nabla^{T^VX})\wedge\ch(\nabla^E),
\label{eq 2.4.4}\\
\lim_{t\to\infty}\ch(\bbb^E_t)&=&\ch(\nabla^{\ker(\DD^{S\otimes E})})
\label{eq 2.4.5}.
\end{eqnarray}
The Bismut--Cheeger eta form \cite{BC89, D91} associated to $\bbb^E$ is defined to be
\begin{equation}
\wt{\eta}^E(g^E, \nabla^E, T^HX, g^{T^VX}):=\int^\infty_0\str\bigg(\frac{d\bbb^E_t}
{dt}e^{-\frac{1}{2\pi i}(\bbb^E_t)^2}\bigg)dt.
\end{equation}
The notation for the Bismut--Cheeger eta form, inspired by Liu \cite{L17}, is to
emphasize the dependence on the geometric data involved.

The local FIT for $\DD^{S\otimes E}$ states that
\begin{equation}
d\wt{\eta}^E(g^E, \nabla^E, T^HX, g^{T^VX})=\int_{X/B}\wh{A}(\nabla^{T^VX})\wedge
\ch(\nabla^E)-\ch(\nabla^{\ker(\DD^{S\otimes E})}).
\end{equation}

Now suppose Assumption \ref{ass 1} is not satisfied. Recall that $B$ is assumed to be
closed. Mi\v s\v cenko--Fomenko \cite{MF79} (see also \cite[Lemma 7.13]{FL10}) prove
that there exist finite rank subbundles $L^\pm\to B$ and complementary closed subbundles
$K^\pm\to B$ of $(\pi^{\spin}_*E)^\pm\to B$ such that
\begin{equation}\label{eq 2.4.8}
(\pi^{\spin}_*E)^+=K^+\oplus L^+,\qquad(\pi^{\spin}_*E)^-=K^-\oplus L^-,
\end{equation}
$\DD^{S\otimes E}_+:(\pi^{\spin}_*E)^+\to(\pi^{\spin}_*E)^-$ is block diagonal as a map
with respect to the direct sum decomposition (\ref{eq 2.4.8}), and $\DD^{S\otimes E}_+
|_{K^+}:K^+\to K^-$ is a smooth bundle isomorphism. Note that the subbundle $K^+\to B$
is not necessarily orthogonal to $L^+\to B$, and the same is true for $K^-\to B$ and
$L^-\to B$.

Given $L^\pm\to B$ satisfying the above conditions, we call the $\Z_2$-graded complex
vector bundle $L\to B$, defined by $L=L^+\oplus L^-$, satisfies the MF property with
respect to $\DD^{S\otimes E}$. If $L\to B$ is a $\Z_2$-graded complex vector bundle
satisfying the MF property with respect to $\DD^{S\otimes E}$, then the $K$-theoretic
analytic index of $[E]\in K(X)$ is defined to be $\ind^{\an}([E])=[L^+]-[L^-]\in K(B)$.
It is proved in \cite[p.96-97]{MF79} that the definition of $\ind^{\an}([E])$ does not
depend on the choice of $L\to B$ satisfying the MF property with respect to
$\DD^{S\otimes E}$.

Given a $\Z_2$-graded complex vector bundle $L\to B$ satisfying the MF property with
respect to $\DD^{S\otimes E}$, define an infinite rank bundle $\wt{\pi^{\spin}_*}E\to B$
by
\begin{displaymath}
\begin{split}
(\wt{\pi^{\spin}_*}E)^+:&=(\pi^{\spin}_*E)^+\oplus L^-,\\
(\wt{\pi^{\spin}_*}E)^-:&=(\pi^{\spin}_*E)^-\oplus L^+.
\end{split}
\end{displaymath}
That is, $\wt{\pi^{\spin}_*}E:=\pi^{\spin}_*E\oplus L^{\op}$ as $\Z_2$-graded complex
vector bundles. Let $i^-:L^-\to(\pi^{\spin}_*E)^-$ be the inclusion map and
$p^+:(\pi^{\spin}_*E)^+\to L^+$ the projection map with respect to (\ref{eq 2.4.8}). For
a given $\alpha\in\C$, define a map $\wt{\DD}^{S\otimes E}_+(\alpha):(\wt{\pi^{\spin}_*}
E)^+\to(\wt{\pi^{\spin}_*}E)^-$ by
$$\wt{\DD}^{S\otimes E}_+(\alpha)=\begin{pmatrix} \DD^{S\otimes E}_+ & \alpha i^- \\
\alpha p^+ & 0 \end{pmatrix}.$$
By \cite[Lemma 7.20]{FL10} $\wt{\DD}^{S\otimes E}_+(\alpha)$ is invertible for all
$\alpha\neq 0$. Define a map $\wt{\DD}^{S\otimes E}(\alpha):\wt{\pi^{\spin}_*}E\to
\wt{\pi^{\spin}_*}E$ by
$$\wt{\DD}^{S\otimes E}(\alpha):=\begin{pmatrix} 0 & (\wt{\DD}^{S\otimes E}_+
(\alpha))^* \\ \wt{\DD}^{S\otimes E}_+(\alpha) & 0 \end{pmatrix}.$$
We refer to \cite[p.943]{FL10} for the properties of $\wt{\DD}^{S\otimes E}(\alpha)$.

The $\Z_2$-graded $L^2$-metric $g^{\pi^{\spin}_*E}$ and the $\Z_2$-graded unitary
connection $\nabla^{\pi^{\spin}_*E, u}$ on $\pi^{\spin}_*E\to B$ project to a
$\Z_2$-graded Hermitian metric $g^L$ and a $\Z_2$-graded unitary connection
$\nabla^L$ on $L\to X$ with respect to (\ref{eq 2.4.8}) respectively. Define a
unitary connection $\nabla^{\wt{\pi_*^{\spin}}E, u}$ on $\wt{\pi_*^{\spin}}E\to B$
by
\begin{displaymath}
\begin{split}
(\nabla^{\wt{\pi_*^{\spin}}E, u})^+&:=(\nabla^{\pi^{\spin}_*E, u})^+\oplus
\nabla^{L^-},\\
(\nabla^{\wt{\pi_*^{\spin}}E, u})^-&:=(\nabla^{\pi^{\spin}_*E, u})^-\oplus
\nabla^{L^+}.
\end{split}
\end{displaymath}
That is, $\nabla^{\wt{\pi_*^{\spin}}E, u}:=\nabla^{\pi^{\spin}_*E, u}\oplus
\nabla^{L, \op}$ as $\Z_2$-graded unitary connections. Note that
$\nabla^{\wt{\pi_*^{\spin}}E, u}$ depends on the choice of $L\to B$ satisfying the MF
property with respect to $\DD^{S\otimes E}$. However, for the clarity of the notation
we will only emphasize the dependence when necessary.

Choose and fix $a\in(0, 1)$ and let $\alpha:[0, \infty)\to[0, 1]$ be a smooth
function such that $\alpha(t)=0$ for all $t\leq a$ and $\alpha(t)=1$ for all
$t\geq 1$. The choice of $a$ is actually immaterial. Define the Bismut
superconnection $\wh{\bbb}^E$ on $\wt{\pi_*^{\spin}}E\to B$ by
\begin{equation}\label{eq 2.4.9}
\wh{\bbb}^E=\wt{\DD}^{S\otimes E}(1)+\nabla^{\wt{\pi_*^{\spin}}E, u}-
\wh{\bbb}^E_{[2]},
\end{equation}
where \dis{\wh{\bbb}^E_{[2]}\in\Omega^2(B, \End^-(\wt{\pi_*^{\spin}}E))} acts on
$\Omega(B, \pi^{\spin}_*E)$ by \dis{\frac{c(T)}{4}}, and acts on $\Omega(B, L^{\op})$
by zero. For convenience we still use the notation \dis{\frac{c(T)}{4}} for
$\wh{\bbb}^E_{[2]}$. The rescaled Bismut superconnection $\wh{\bbb}^E_t$ is given by
$$\wh{\bbb}^E_t=\sqrt{t}\wt{\DD}^{S\otimes E}(\alpha(t))+\nabla^{\wt{\pi_*^{\spin}}
E, u}-\frac{c(T)}{4\sqrt{t}}.$$
Since $\wt{\DD}^{S\otimes E}(\alpha(t))$ is invertible for $t$ sufficiently large
(more precisely, for $t\geq 1$), we have \cite[(7.24)]{FL10}:
\begin{equation}\label{eq 2.4.10}
\lim_{t\to\infty}\ch(\wh{\bbb}^E_t)=0.
\end{equation}
\begin{remark}\label{remark 4}
As noted in \cite[p.943]{FL10}, when $t\to 0$ (more precisely, for $t\leq a$) the
rescaled Bismut superconnection $\wh{\bbb}^E_t$ decouples, i.e.
$$\wh{\bbb}^E_t=\bigg(\sqrt{t}\DD^{S\otimes E}+\nabla^{\pi^{\spin}_*E, u}-\frac{c(T)}
{4\sqrt{t}}\bigg)\oplus\nabla^{L, \op}=\bbb^E_t\oplus\nabla^{L, \op}.$$
\end{remark}
It follows from Remark \ref{remark 4} that we have \cite[(7.23)]{FL10}:
\begin{equation}\label{eq 2.4.11}
\lim_{t\to 0}\ch(\wh{\bbb}^E_t)=\lim_{t\to 0}\ch(\bbb^E_t)-\ch(\nabla^L)=\int_{X/B}
\wh{A}(\nabla^{T^VX})\wedge\ch(\nabla^E)-\ch(\nabla^L).
\end{equation}
The Bismut--Cheeger eta form associated to $\wh{\bbb}^E$ \cite[(7.25)]{FL10} is defined to
be
$$\wh{\eta}^E(g^E, \nabla^E, T^HX, g^{T^VX}, L)=\int^\infty_0\str\bigg(\frac{d\wh{\bbb}^E_t}
{dt}e^{-\frac{1}{2\pi i}(\wh{\bbb}^E_t)^2}\bigg)dt.$$
Note that $\wh{\eta}^E(g^E, \nabla^E, T^HX, g^{T^VX}, L)$ does not depend on the choice of
$\alpha$.

The local FIT for $\DD^{S\otimes E}$ without Assumption \ref{ass 1} is given by the local
FIT for $\wt{\DD}^{S\otimes E}(1)$ \cite[(7.26)]{FL10}, i.e.
\begin{equation}\label{eq 2.4.12}
d\wh{\eta}^E(g^E, \nabla^E, T^HX, g^{T^VX}, L)=\int_{X/B}\wh{A}(\nabla^{T^VX})\wedge
\ch(\nabla^E)-\ch(\nabla^L).
\end{equation}
\begin{remark}\label{remark 5}
The purpose of putting the projected $\Z_2$-graded unitary connection $\nabla^L$ on $L
\to B$ \cite[p.943]{FL10} is to obtain the local FIT for $\wt{\DD}^{S\otimes E}(1)$.
However, for our purpose we will choose other $\Z_2$-graded unitary connection on $L\to B$.
More precisely, let $\wt{\nabla}^L$ be a $\Z_2$-graded unitary connection on $L\to B$, not
necessarily projected from $\nabla^{\pi^{\spin}_*E, u}$. Then the rescaled Bismut
superconnection
$$\wh{\bbb}^E_t:=\sqrt{t}\wt{\DD}^{S\otimes E}(\alpha(t))+(\nabla^{\pi^{\spin}_*E, u}
\oplus\wt{\nabla}^{L, \op})-\frac{c(T)}{4\sqrt{t}}$$
still satisfies the analogs of (\ref{eq 2.4.10}) and (\ref{eq 2.4.11}). The corresponding
Bismut--Cheeger eta form can still be defined, but it depends on $\wt{\nabla}^L$. The
analog of (\ref{eq 2.4.12}) still holds. Henceforth we write
$$\wh{\eta}^E(g^E, \nabla^E, T^HX, g^{T^VX}, L, \wt{\nabla}^L)$$
for the corresponding Bismut--Cheeger eta form when the unitary connection $\wt{\nabla}^L$
on $L\to B$ is not projected from $\pi^{\spin}_*E\to B$.
\end{remark}

\subsection{The analytic index in differential $K$-theory}

In this subsection we review the definition of Freed--Lott differential $K$-theory and the
analytic index in differential $K$-theory defined under and without Assumption \ref{ass 1}
respectively \cite{FL10}. The setup of the local FIT in \cite{FL10} is defined for
spin$^c$ fibers, but for our purpose we need spin fibers.

The Freed--Lott differential $K$-group $\wh{K}_{\FL}(X)$ is the abelian group generated
by quadruples $\E=(E, g^E, \nabla^E, \omega)$, where $E\to X$ is a complex vector bundle
with a Hermitian metric $g^E$ and a unitary connection $\nabla^E$, and \dis{\omega\in
\frac{\Omega^{\odd}(X)}{\im(d)}}. The only relation is $\E_0=\E_1$ if and only if there
exists a generator $(F, g^F, \nabla^F, \omega^F)$ of $\wh{K}_{\FL}(X)$ such that
\begin{eqnarray}
E_0\oplus F&\cong& E_1\oplus F,\label{eq 2.5.1}\\
\omega_0-\omega_1&=&\CS(\nabla^{E_0}\oplus\nabla^F, \nabla^{E_1}\oplus\nabla^F)\textrm{ in }
\frac{\Omega^{\odd}(X)}{\im(d)}.\label{eq 2.5.2}
\end{eqnarray}
In (\ref{eq 2.5.2}), instead of writing $F^*(\nabla^{E_1}\oplus\nabla^F)$, where
$F:E_0\oplus F\to E_1\oplus F$ is a smooth bundle isomorphism given by (\ref{eq 2.5.1}),
we follow the convention in \cite{FL10} that $F^*$ is suppressed.\footnote{This convention
will also be applied to Corollary \ref{coro 3.1.1} and \ref{coro 3.1.2}.} The differential
$K$-group $\wh{K}_{\FL}(X)$ can also be described in terms of $\Z_2$-graded generators,
i.e. $\E=(E, g^E, \nabla^E, \omega)$, where $E\to X$, $g^E$ and $\nabla^E$ are
$\Z_2$-graded. Two $\Z_2$-graded generators $\E_0$ and $\E_1$ are equal in $\wh{K}_{\FL}(X)$
if and only if there exist $\Z_2$-graded generators $\W=(W, g^W, \nabla^W, \omega_W)$ and
$\V=(V, g^V, \nabla^V, \omega_V)$ of $\wh{K}_{\FL}(X)$ of the form
\begin{equation}\label{eq 2.5.3}
W^+=W^-, g^{W^+}=g^{W^-}, \nabla^{W^+}=\nabla^{W^-},
\end{equation}
and similarly for $\V$, such that
\begin{eqnarray}
E_0\oplus W&\cong& E_1\oplus V\textrm{ as } \Z_2 \textrm{-graded complex vector bundles},
\label{eq 2.5.4}\\
\omega_0-\omega_1&=&\CS(\nabla^{E_0}\oplus\nabla^W, \nabla^{E_1}\oplus\nabla^V)
\textrm{ in }\frac{\Omega^{\odd}(X)}{\im(d)}.\label{eq 2.5.5}
\end{eqnarray}

Let $\pi:X\to B$ be a submersion with closed, oriented and spin fibers of even
dimension, and $\E=(E, g^E, \nabla^E, \omega)$ a generator of $\wh{K}_{\FL}(X)$. If
Assumption \ref{ass 1} is satisfied, then the differential analytic index
$\ind^{\an}_{\FL}(\E)\in\wh{K}_{\FL}(B)$ \cite[Definition 3.12]{FL10} is defined to
be
\begin{equation}\label{eq 2.5.6}
\begin{split}
\ind^{\an}_{\FL}(\E)&=\bigg(\ker(\DD^{S\otimes E}), g^{\ker(\DD^{S\otimes E})},
\nabla^{\ker(\DD^{S\otimes E})},\\
&\qquad\int_{X/B}\wh{A}(\nabla^{T^VX})\wedge\omega+\wt{\eta}^E(g^E, \nabla^E, T^HX,
g^{T^VX})\bigg).
\end{split}
\end{equation}

If Assumption \ref{ass 1} is not satisfied, the differential analytic index
$\ind^{\an}_{\FL}(\E; L)\in\wh{K}_{\FL}(B)$ \cite[Definition 7.27]{FL10} is defined
to be
\begin{equation}\label{eq 2.5.7}
\ind^{\an}_{\FL}(\E; L)=\bigg(L, g^L, \nabla^L, \int_{X/B}\wh{A}(\nabla^{T^VX})
\wedge\omega+\wh{\eta}^E(g^E, \nabla^E, T^HX, g^{T^VX}, L)\bigg),
\end{equation}
where $L\to B$ is a fixed choice of $\Z_2$-graded complex vector bundle satisfying the
MF property with respect to $\DD^{S\otimes E}$, and $g^L$, $\nabla^L$ are the projected
$\Z_2$-graded Hermitian metric and projected $\Z_2$-graded unitary connection on $L\to B$.

\section{Main results}

In this section we prove the main results of this paper. Henceforth by geometric data
we mean the quadruple
$$(g^E, \nabla^E, T^HX, g^{T^VX})$$
as defined in Section \ref{s 2.3}.

\subsection{A variational formula of the Bismut--Cheeger eta form without the kernel
bundle assumption and its applications}\label{s 3.1}

In this subsection we prove a variational formula of the Bismut--Cheeger eta form without
Assumption \ref{ass 1} (Proposition \ref{prop 3.1.1}).

The following lemma, which is an immediate consequence of \cite[Theorem 2.10]{B96}, roughly
says the Bismut--Cheeger eta form defined without Assumption \ref{ass 1} is stable under
perturbation of split quadruple.
\begin{lemma}\label{lemma 3.1.1}
Let $\pi:X\to B$ be a submersion with closed, oriented and spin fibers of even dimension
and $E\to X$ a complex vector bundle. Denote by $(g^E, \nabla^E, T^HX, g^{T^VX})$ a fixed
choice of geometric data. Let $L\to B$ be a $\Z_2$-graded complex vector bundle satisfying
the MF property with respect to $\DD^{S\otimes E}$. Denote by $g^L$ the projected Hermitian
metric and $\nabla^L$ the projected $\Z_2$-graded unitary connection on $L\to B$. Let
$\alpha:(0, \infty)\to[0, 1]$ be the smooth function given in Section \ref{s 2.4} and
$\wh{\bbb}^E$ the Bismut superconnection on $\wt{\pi^{\spin}_*}E\to B$ defined by
(\ref{eq 2.4.9}). Let $(W, g^W, \nabla^W, s_W)$ be a split quadruple over $B$ (cf. Section
\ref{s 2.1}). Define a superconnection $\aaa^W$ on $W\to B$ by
\begin{equation}\label{eq 3.1.1}
\aaa^W=s_W+\nabla^W,
\end{equation}
and define the rescaled superconnection $\aaa^W_t$ by
\begin{equation}\label{eq 3.1.2}
\aaa^W_t=\sqrt{t}\alpha(t)s_W+\nabla^W.
\end{equation}
If $\wh{\bbb}^{E, W}$ is the Bismut superconnection on $\wt{\pi^{\spin}_*}E\oplus W\to B$
defined by
$$\wh{\bbb}^{E, W}=\wh{\bbb}^E\oplus\aaa^W,$$
then
\begin{equation}\label{eq 3.1.3}
\wh{\eta}^{E, W}(g^E, \nabla^E, T^HX, g^{T^VX}, L\oplus W, \nabla^L\oplus\nabla^W)=
\wh{\eta}^E(g^E, \nabla^E, T^HX, g^{T^VX}, L)
\end{equation}
in \dis{\frac{\Omega^{\odd}(B)}{\im(d)}}, where the left-hand side of (\ref{eq 3.1.3})
denotes the Bismut--Cheeger eta form associated to $\wh{\bbb}^{E, W}$.
\end{lemma}
\begin{proof}
Note that for all $t\in(0, \infty)$ we have $\wh{\bbb}^{E, W}_t=\wh{\bbb}^E_t\oplus
\aaa^W_t$. By (\ref{eq 2.1.12}), for any $t<T\in(0, \infty)$ we have
\begin{displaymath}
\begin{split}
\CS(\wh{\bbb}^{E, W}_T, \wh{\bbb}^{E, W}_t)&=\CS(\wh{\bbb}^E_T\oplus\aaa^W_T,
\wh{\bbb}^E_t\oplus\aaa^W_t)\\
&=\CS(\wh{\bbb}^E_T, \wh{\bbb}^E_t)+\CS(\aaa^W_T, \aaa^W_t)
\end{split}
\end{displaymath}
in \dis{\frac{\Omega^{\odd}(B)}{\im(d)}}. Denote by $\wt{\eta}^W$ the Bismut--Cheeger
eta form associated to $\aaa^W$. By letting $T\to\infty$ and $t\to 0$ in above we have
\begin{displaymath}
\begin{split}
&\wh{\eta}^{E, W}(g^E, \nabla^E, T^HX, g^{T^VX}, L\oplus W, \nabla^L\oplus\nabla^W)\\
=~&\wh{\eta}^E(g^E, \nabla^E, T^HX, g^{T^VX}, L)+\wt{\eta}^W.
\end{split}
\end{displaymath}
Thus proving (\ref{eq 3.1.3}) is equivalent to proving \dis{\wt{\eta}^W=0\in
\frac{\Omega^{\odd}(B)}{\im(d)}}. Since the quadruple $(W, g^W, \nabla^W, s_W)$ splits,
it follows from \cite[(2.23)]{B96} that
$$(\aaa^W_t)^2=t\alpha(t)^2P^{\im(s_W)}+(\nabla^W)^2,$$
where $P^{\im(s_W)}:W\to\im(s_W)$ is the projection onto $\im(s_W)\to B$ with respect to
the direct sum decomposition $W=\ker(s_W)\oplus\im(s_W)$. By \cite[(2.24)]{B96} we have
$$\str\bigg(\frac{d\aaa^W_t}{dt}e^{-\frac{1}{2\pi i}(\aaa^W_t)^2}\bigg)=0.$$
Thus $\wt{\eta}^W=0\in\Omega^{\odd}(B)$, and therefore (\ref{eq 3.1.3}) holds.
\end{proof}

Recall from \cite[p.289]{BGV} that a $\Z_2$-graded complex vector bundle $E\to X$ defines
an element in $K(X)$ by $E\mapsto[E^+]-[E^-]$. Two $\Z_2$-graded complex vector bundles
$E\to X$ and $F\to X$ define the same element in $K(X)$ if and only if there exist two
complex vector bundles $G\to X$ and $H\to X$ such that
\begin{equation}\label{eq 3.1.4}
\begin{split}
E^+\oplus G&\cong F^+\oplus H,\\
E^-\oplus G&\cong F^-\oplus H.
\end{split}
\end{equation}

We now prove a variational formula of the Bismut--Cheeger eta form without Assumption
\ref{ass 1}. The proof is actually similar to the ones of \cite{H16, L17}. The extra
technicality is caused by different choices of $\Z_2$-graded complex vector bundle
satisfying the MF property with respect to $\DD^{S\otimes E}$.
\begin{prop}\label{prop 3.1.1}
Let $\pi:X\to B$ be a submersion with closed, oriented and spin fibers of even dimension
$n$, and $E\to X$ a complex vector bundle. Fix $k\in\set{0, 1}$. Denote by $(g^E_k,
\nabla^E_k, T^H_kX, g^{T^VX}_k)$ the geometric data, $\DD^{S\otimes E}_k$ the corresponding
twisted spin Dirac operator, and $L_k\to B$ a $\Z_2$-graded complex vector bundle
satisfying the MF property with respect to $\DD^{S\otimes E}_k$. Denote by $g^{L_k}$ and
$\nabla^{L_k}$ the $\Z_2$-graded Hermitian metric and the $\Z_2$-graded unitary connection
on $L_k\to B$ projected from $\pi^{\spin}_*E\to B$. Then for $k\in\set{0, 1}$, there exist
a $\Z_2$-graded complex vector bundle $W_k\to B$ of the form $W_k^+=W_k^-$ with a
$\Z_2$-graded Hermitian metric $g^{W_k}$ of the form $g^{W_k^+}=g^{W_k^-}$ and a
$\Z_2$-graded unitary connection $\nabla^{W_k}$ of the form $\nabla^{W_k^+}=\nabla^{W_k^-}$
such that
$$L_0\oplus W_0\cong L_1\oplus W_1$$
as $\Z_2$-graded complex vector bundles via a smooth $\Z_2$-graded bundle isomorphism
$h:L_1\oplus W_1\to L_0\oplus W_0$, and
\begin{equation}\label{eq 3.1.5}
\begin{split}
&\wh{\eta}^E(g^E_1, \nabla^E_1, T^H_1X, g^{T^VX}_1, L_1)-\wh{\eta}^E(g^E_0, \nabla^E_0,
T^H_0X, g^{T^VX}_0, L_0)\\
=&\int_{X/B}\wt{\wh{A}}(\nabla^{T^VX}_0, \nabla^{T^VX}_1)\wedge\ch(\nabla^E_0)+\int_{X/B}
\wh{A}(\nabla^{T^VX}_1)\wedge\CS(\nabla^E_0, \nabla^E_1)\\
&\quad-\CS(h^*(\nabla^{L_0}\oplus\nabla^{W_0}), \nabla^{L_1}\oplus\nabla^{W_1})
\end{split}
\end{equation}
in \dis{\frac{\Omega^{\odd}(B)}{\im(d)}}, where $\wt{\wh{A}}(\nabla^{T^VX}_0,
\nabla^{T^VX}_1)$ is defined by (\ref{eq 2.1.16}).
\end{prop}
\begin{proof}
Since the space of the splitting map is affine, there exists a smooth path of horizontal
distributions $\set{T^H_tX\to X}_{t\in[0, 1]}$ joining $T^H_0X\to X$ and $T^H_1X\to X$.
By Section \ref{s 2.1} and above, there exists a smooth path
\begin{equation}\label{eq 3.1.6}
(g^E_t, \nabla^E_t, T^H_tX, g^{T^VX}_t)\textrm{ with } t\in[0, 1],
\end{equation}
joining $(g^E_0, \nabla^E_0, T^H_0X, g^{T^VX}_0)$ and $(g^E_1, \nabla^E_1, T^H_1X,
g^{T^VX}_1)$. From (\ref{eq 3.1.6}) one can define a new path, denoted by $\alpha$,
joining $(g^E_0, \nabla^E_0, T^H_0X, g^{T^VX}_0)$ and $(g^E_1, \nabla^E_1, T^H_1X,
g^{T^VX}_1)$, by
\begin{equation}\label{eq 3.1.7}
\alpha(t)=\left\{
            \begin{array}{ll}
              (g^E_0, \nabla^E_0, T^H_{2t}X, g^{T^VX}_{2t}), & \textrm{ for } t\in
              [0, \frac{1}{2}],  \\
              (g^E_{2t-1}, \nabla^E_{2t-1}, T^H_1X, g^{T^VX}_1), & \textrm{ for } t
              \in [\frac{1}{2}, 1]
            \end{array}
          \right..
\end{equation}
Here, for $t\in[0, \frac{1}{2}]$, the path $(T^H_{2t}X, g^{T^VX}_{2t})$ joining $(T^H_0X,
g^{T^VX}_0)$ and $(T^H_1X, g^{T^VX}_t)$ is induced by (\ref{eq 3.1.6}); for $t\in
[\frac{1}{2}, 1]$, the path $(g^E_{2t-1}, \nabla^E_{2t-1})$ joining $(g^E_0, \nabla^E_0)$
and $(g^E_1, \nabla^E_1)$ is induced by (\ref{eq 3.1.6}).

Consider the following diagram
\cdd{\e & & E \\ @VVV @VVV \\ \wt{X} @>>p_X> X \\ @V\wt{\pi}VV @VV\pi V \\ \wt{B}
@>>p_B> B}
where $\e:=p_X^*E$. The smooth path (\ref{eq 3.1.7}) defines the geometric data
\begin{equation}\label{eq 3.1.8}
(g^{\e}, \nabla^{\e}, T^H\wt{X}, g^{T^V\wt{X}}),
\end{equation}
where $\nabla^{\e}$ is defined by (\ref{eq 2.1.6}). Since the fibers of $\pi:X\to B$
are oriented and spin, the same is true for the fibers of $\wt{\pi}:\wt{X}\to\wt{B}$.

Denote by $\DD^{S\otimes\e}:\wt{\pi}^{\spin}_*\e\to\wt{\pi}^{\spin}_*\e$ the twisted spin
Dirac operator defined by the geometric data (\ref{eq 3.1.8}). Since Assumption \ref{ass 1}
is not satisfied, we choose and fix a $\Z_2$-graded complex vector bundle $\LLL\to\wt{B}$
satisfying the MF property with respect to $\DD^{S\otimes\e}$. Denote by $\K\to\wt{B}$ a
$\Z_2$-graded complementary subbundle of $\wt{\pi}^{\spin}_*\e\to\wt{B}$, i.e.
\begin{equation}\label{eq 3.1.9}
(\wt{\pi}_*^{\spin}\e)^\pm=\K^\pm\oplus\LLL^\pm.
\end{equation}
Denote by $g^{\LLL}$ the projected $\Z_2$-graded Hermitian metric on $\LLL\to\wt{B}$.

As in (\ref{eq 2.1.2}) we have the following smooth bundle isomorphisms
\begin{equation}\label{eq 3.1.10}
\begin{split}
i_{B, 0}^*(\wt{\pi}^{\spin}_*\e)^\pm&\cong i_{B, 1}^*(\wt{\pi}^{\spin}_*\e)^\pm,\\
i_{B, 0}^*\K^\pm&\cong i_{B, 1}^*\K^\pm,\\
i_{B, 0}^*\LLL^\pm&\cong i_{B, 1}^*\LLL^\pm.
\end{split}
\end{equation}
Write $K^\pm\to B$ for $i_{B, 0}^*\K^\pm\to B$ and $L^\pm\to B$ for $i_{B, 0}^*\LLL^\pm
\to B$. Also, write $g^{L_k}=i_{B, k}^*g^{\LLL}$ for $k\in\set{0, 1}$. On the other hand,
since
$$i_{B, k}^*(\wt{\pi}^{\spin}_*\e)^\pm\cong(\pi^{\spin}_*E)^\pm,$$
for $k=0$ and $k=1$ respectively, it follows that $\pi^{\spin}_*E\to B$ admits a direct
sum decomposition, given by
\begin{equation}\label{eq 3.1.11}
(\pi^{\spin}_*E)^+=K^+\oplus L^+,\qquad(\pi^{\spin}_*E)^-=K^-\oplus L^-.
\end{equation}
Since
\begin{equation}\label{eq 3.1.12}
\DD^{S\otimes\e}|_{i_{B, k}^*\wt{\pi}^{\spin}_*\e}=\DD^{S\otimes E}_k,
\end{equation}
it follows that $\DD^{S\otimes E}_k:(\pi^{\spin}_*E)^+\to(\pi^{\spin}_*E)^-$ is block
diagonal with respect to (\ref{eq 3.1.11}) and the restriction $\DD^{S\otimes E}_k:K^+
\to K^-$ is an isomorphism. Thus $L\to B$ satisfies the MF property with respect to
$\DD^{S\otimes E}_k$ for $k=0$ and $k=1$ respectively. Therefore the $K$-theoretic
analytic index of $[E]\in K(X)$ is given by
$$\ind^{\an}([E])=[L^+]-[L^-].$$
By assumption, the $K$-theoretic analytic index of $[E]\in K(X)$ are given by
$$\ind^{\an}([E])=[L_0^+]-[L_0^-]\textrm{ and }\ind^{\an}([E])=[L_1^+]-[L_1^-]$$
respectively. Since the $K$-theoretic analytic index is well defined, it follows from
(\ref{eq 3.1.3}) that there exist complex vector bundles $G_k\to B$ and $H_k\to B$, where
$k\in\set{0, 1}$, such that
\begin{equation}\label{eq 3.1.13}
\begin{split}
L^+_0\oplus G_0&\cong L^+\oplus H_0,\\
L^-_0\oplus G_0&\cong L^-\oplus H_0,\\
L^+_1\oplus G_1&\cong L^+\oplus H_1,\\
L^-_1\oplus G_1&\cong L^-\oplus H_1.
\end{split}
\end{equation}
It follows from (\ref{eq 3.1.13}) that
\begin{equation}\label{eq 3.1.14}
\begin{split}
L^+_0\oplus(G_0\oplus H_1)&\cong L^+\oplus H_0\oplus H_1\cong L^+\oplus H_1\oplus
H_0\cong L^+_1\oplus(G_1\oplus H_0),\\
L^-_0\oplus(G_0\oplus H_1)&\cong L^-\oplus H_0\oplus H_1\cong L^-\oplus H_1\oplus
H_0\cong L^-_1\oplus(G_1\oplus H_0).
\end{split}
\end{equation}
Write
\begin{displaymath}
\begin{split}
H^+&=H^-=H_0\oplus H_1,\\
W_0^+&=W_0^-=G_0\oplus H_1,\\
W_1^+&=W_1^-=G_1\oplus H_0
\end{split}
\end{displaymath}
respectively. Define $\Z_2$-graded complex vector bundles $H\to B$, $W_0\to B$ and $W_1
\to B$ by
$$H=H^+\oplus H^-,\qquad W_0=W_0^+\oplus W_0^-,\qquad W_1=W_1^+\oplus W_1^-.$$
Write $\HH=p_B^*H$. By (\ref{eq 3.1.14}) we have the following commutative diagram
\begin{equation}\label{eq 3.1.15}
\xymatrix{ & L\oplus H \ar[dl]_{f_0} \ar[dr]^{f_1} & \\ L_0\oplus W_0  & & L_1\oplus W_1
\ar[ll]^h}
\end{equation}
where $f_0$ and $f_1$ are the resulting smooth $\Z_2$-graded bundle isomorphisms, and
$h:=f_0\circ f_1^{-1}$.

For $k\in\set{0, 1}$, put a Hermitian metric $g^{W_k^+}$ on $W_k^+\to B$ and a unitary
connection $\nabla^{W_k^+}$ on $W_k^+\to B$ with respect to $g^{W_k^+}$. Then define a
$\Z_2$-graded Hermitian metric $g^{W_k}$ and a $\Z_2$-graded unitary connection
$\nabla^{W_k}$ on $W_k\to B$ by $g^{W_k^-}:=g^{W_k^+}$ and $\nabla^{W_k^-}:=
\nabla^{W_k^+}$.

Note that $g^{L\oplus H}_0:=f_0^*(g^{L_0}\oplus g^{W_0})$ and $g^{L\oplus H}_1:=f_1^*
(g^{L_1}\oplus g^{W_1})$ are $\Z_2$-graded Hermitian metrics on $L\oplus H\to B$, and
$$\nabla_0:=f_0^*(\nabla^{L_0}\oplus\nabla^{W_0}),\qquad\nabla_1:=f_1^*(\nabla^{L_1}
\oplus\nabla^{W_1})$$
are $\Z_2$-graded connections on $L\oplus H\to B$ that are unitary with respect to
$g^{L\oplus H}_0$ and $g^{L\oplus H}_1$ respectively. Define a $\Z_2$-graded unitary
connection $\wt{\nabla}^{\LLL\oplus\HH}$ on $\LLL\oplus\HH\to\wt{B}$ by (\ref{eq 2.1.6})
such that
\begin{equation}\label{eq 3.1.16}
i_{B, 0}^*\wt{\nabla}^{\LLL\oplus\HH}=\nabla_0,\qquad i_{B, 1}^*\wt{\nabla}^{\LLL\oplus
\HH}=\nabla_1.
\end{equation}

The rescaled Bismut superconnection $\wh{\bbb}^{\e, \HH}_t$ on the $\Z_2$-graded infinite
rank bundle $\wt{\pi}^{\spin}_*\e\oplus\LLL^{\op}\oplus\HH^{\op}\to\wt{B}$ is defined to be
$$\wh{\bbb}^{\e, \HH}_t=\wh{\bbb}^{\e, \HH}_{[0], t}+(\nabla^{\wt{\pi}^{\spin}_*\e, u}
\oplus\wt{\nabla}^{\LLL\oplus\HH, \op})-\wh{\bbb}^{\e, \HH}_{[2], t},$$
where
$$\wh{\bbb}^{\e, \HH}_{[0], t}=\sqrt{t}\begin{pmatrix} 0 & 0 &
(\wt{\DD}^{S\otimes\e}_+(\alpha(t)))^* & 0 \\ 0 & 0 & 0 & \alpha(t)\id \\
\wt{\DD}^{S\otimes\e}_+(\alpha(t)) & 0 & 0 & 0 \\ 0 & \alpha(t)\id & 0 & 0
\end{pmatrix},$$
and $\wh{\bbb}^{\e, \HH}_{[2], t}\in\Omega^2(\wt{B}, \End^-(\wt{\pi}^{\spin}_*\e\oplus
\LLL^{\op}\oplus\HH^{\op}))$ acts on $\Omega(\wt{B}, \wt{\pi}^{\spin}_*\e)$ by
\dis{\frac{c(\wt{T})}{4\sqrt{t}}} and acts on $\Omega(\wt{B}, \LLL^{\op}\oplus\HH^{\op})$
by zero. Here $\wt{T}$ is the curvature 2-form of the fiber bundle $\wt{\pi}:\wt{X}\to
\wt{B}$. For convenience we still use the notation \dis{\frac{c(\wt{T})}{4\sqrt{t}}} for
$\wh{\bbb}^{\e, \HH}_{[2], t}$. Thus
$$\wh{\bbb}^{\e, \HH}_t=\wh{\bbb}^{\e, \HH}_{[0], t}+(\nabla^{\wt{\pi}^{\spin}_*\e, u}
\oplus\wt{\nabla}^{\LLL\oplus\HH, \op})-\frac{c(\wt{T})}{4\sqrt{t}}.$$
The (unrescaled) Bismut superconnection $\wh{\bbb}^{\e, \HH}$ is recovered by taking
$t=1$ in $\wh{\bbb}^{\e, \HH}_t$.

By Remark \ref{remark 4}, for $t\leq a$ we have
$$\wh{\bbb}^{\e, \HH}_t=\bbb^{\e}_t\oplus\wt{\nabla}^{\LLL\oplus\HH, \op}.$$
Thus by Remark \ref{remark 5} the analog of (\ref{eq 2.4.11}) holds, i.e.
\begin{equation}\label{eq 3.1.17}
\begin{split}
\lim_{t\to 0}\ch(\wh{\bbb}^{\e, \HH}_t)&=\lim_{t\to 0}\ch(\bbb^{\e}_t)-\ch
(\nabla^{\LLL\oplus\HH})\\
&=\int_{\wt{X}/\wt{B}}\wh{A}(\nabla^{T^V\wt{X}})\wedge\ch(\nabla^{\e})-\ch
(\wt{\nabla}^{\LLL\oplus\HH}).
\end{split}
\end{equation}
On the other hand, since $\wh{\bbb}^{\e, \HH}_{[0], t}$ is invertible for $t\geq 1$,
it follows that the analog of (\ref{eq 2.4.10}) holds, i.e.
\begin{equation}\label{eq 3.1.18}
\lim_{t\to\infty}\ch(\wh{\bbb}^{\e, \HH}_t)=0.
\end{equation}
Define the Bismut--Cheeger eta form associated to $\wh{\bbb}^{\e, \HH}$ by
$$\wh{\eta}^{\e, \HH}(g^{\e}, \nabla^{\e}, T^H\wt{X}, g^{T^V\wt{X}}, \LLL\oplus\HH,
\wt{\nabla}^{\LLL\oplus\HH})=\int^\infty_0\str\bigg(\frac{d\wh{\bbb}^{\e, \HH}_t}
{dt}e^{-\frac{1}{2\pi i}(\wh{\bbb}^{\e, \HH}_t)^2}\bigg)dt.$$
We now temporarily suppress the data defining the Bismut--Cheeger eta form to shorten the
expression. By (\ref{eq 3.1.17}), (\ref{eq 3.1.18}) and Remark \ref{remark 5} we have
\begin{equation}\label{eq 3.1.19}
d\wh{\eta}^{\e, \HH}=\int_{\wt{X}/\wt{B}}\wh{A}(\nabla^{T^V\wt{X}})\wedge\ch(\nabla^{\e})-
\ch(\wt{\nabla}^{\LLL\oplus\HH}).
\end{equation}
Denote by $i:\partial\wt{B}\to\wt{B}$ the inclusion map. By taking $M\to B$ to be
$\wt{B}\to B$ and $\omega$ to be $\wh{\eta}^{\e, \HH}$ in (\ref{eq 2.1.4}) we have
\begin{equation}\label{eq 3.1.20}
-(i_{B, 1}^*\wh{\eta}^{\e, \HH}-i_{B, 0}^*\wh{\eta}^{\e, \HH})=-\int_{\partial\wt{B}/B}
i^*\wh{\eta}^{\e, \HH}=\int_{\wt{B}/B}d^{\wt{B}}\wh{\eta}^{\e, \HH}-d^B\int_{\wt{B}/B}
\wh{\eta}^{\e, \HH}.
\end{equation}
By modding out exact forms in (\ref{eq 3.1.20}), it follows from (\ref{eq 2.1.3}) and
(\ref{eq 3.1.19}) that
\begin{equation}\label{eq 3.1.21}
\begin{split}
&i_{B, 1}^*\wh{\eta}^{\e, \HH}-i_{B, 0}^*\wh{\eta}^{\e, \HH}=-\int_{\wt{B}/B}d^{\wt{B}}
\wh{\eta}^{\e, \HH}\\
=&-\int_{\wt{B}/B}\bigg(\int_{\wt{X}/\wt{B}}\wh{A}(\nabla^{T^V\wt{X}})\wedge
\ch(\nabla^{\e})-\ch(\wt{\nabla}^{\LLL\oplus\HH})\bigg)\\
=&\int_{\wt{B}/B}\int_{\wt{X}/\wt{B}}(-\wh{A}(\nabla^{T^V\wt{X}})\wedge\ch(\nabla^{\e}))
-\CS(\nabla_0, \nabla_1)
\end{split}
\end{equation}
in \dis{\frac{\Omega^{\odd}(B)}{\im(d)}}. Since $f_1$ covers the identity map $\id_B$,
it follows that
\begin{displaymath}
\begin{split}
\CS(\nabla_0, \nabla_1)&=\CS(f_0^*(\nabla^{L_0}\oplus\nabla^{W_0}), f_1^*(\nabla^{L_1}
\oplus\nabla^{W_1}))\\
&=\CS((f_1^{-1})^*f_0^*(\nabla^{L_0}\oplus\nabla^{W_0}), \nabla^{L_1}\oplus\nabla^{W_1})
\end{split}
\end{displaymath}
Note that $h^*=(f_1^{-1})^*\circ f_0^*$. Since \dis{\int_{X/B}\circ\int_{\wt{X}/X}=
\int_{\wt{X}/B}=\int_{\wt{B}/B}\circ\int_{\wt{X}/\wt{B}}}, it follows that
\begin{equation}\label{eq 3.1.22}
\int_{\wt{B}/B}\int_{\wt{X}/\wt{B}}(-\wh{A}(\nabla^{T^V\wt{X}})\wedge\ch(\nabla^{\e}))
=\int_{X/B}\int_{\wt{X}/X}(-\wh{A}(\nabla^{T^V\wt{X}})\wedge\ch(\nabla^{\e})).
\end{equation}
By (\ref{eq 3.1.8}), (\ref{eq 2.1.3}) and (\ref{eq 2.1.16}) we have
\begin{equation}\label{eq 3.1.23}
\begin{split}
\int_{\wt{X}/X}(-\wh{A}(\nabla^{T^V\wt{X}})\wedge\ch(\nabla^{\e}))&=\wt{\wh{A}}
(\nabla^{T^VX}_0, \nabla^{T^VX}_1)\wedge\ch(\nabla^E_0)\\
&\qquad+\wh{A}(\nabla^{T^VX}_1)\wedge\CS(\nabla^E_0, \nabla^E_1).
\end{split}
\end{equation}
By putting (\ref{eq 3.1.22}) and (\ref{eq 3.1.23}) into (\ref{eq 3.1.21}) we obtain
\begin{equation}\label{eq 3.1.24}
\begin{split}
&i_{B, 1}^*\wh{\eta}^{\e, \HH}-i_{B, 0}^*\wh{\eta}^{\e, \HH}\\
=&\int_{X/B}\wt{\wh{A}}(\nabla^{T^VX}_0, \nabla^{T^VX}_1)\wedge\ch(\nabla^E_0)+
\int_{X/B}\wh{A}(\nabla^{T^VX}_1)\wedge\CS(\nabla^E_0, \nabla^E_1)\\
&\qquad-\CS(h^*(\nabla^{L_0}\oplus\nabla^{W_0}), \nabla^{L_1}\oplus\nabla^{W_1})
\end{split}
\end{equation}
in \dis{\frac{\Omega^{\odd}(B)}{\im(d)}}. On the other hand, it follows from
(\ref{eq 3.1.12}), (\ref{eq 3.1.15}) and (\ref{eq 3.1.16}) that
\begin{equation}\label{eq 3.1.25}
\begin{split}
&i_{B, 1}^*\wh{\eta}^{\e, \HH}(g^{\e}, \nabla^{\e}, T^H\wt{X}, g^{T^V\wt{X}},
\LLL\oplus\HH, \wt{\nabla}^{\LLL\oplus\HH})\\
=~&\wh{\eta}^{E, W}(g^E_1, \nabla^E_1, T^H_1X, g^{T^VX}_1, L_1\oplus W, \nabla^{L_1}
\oplus\nabla^{W_1}).
\end{split}
\end{equation}
Define $s_{W_1}\in\Gamma(B, \End^-(W_1))$ by \dis{s_{W_1}=\begin{pmatrix} 0 & \id \\
\id & 0 \end{pmatrix}}. By Example \ref{exam 1} the quadruple $(W_1, g^{W_1}, \nabla^{W_1},
s_{W_1})$ over $B$ splits. Define a superconnection $\aaa^{W_1}$ on $W_1\to B$ by
(\ref{eq 3.1.1}) and its rescaled superconnection $\aaa^{W_1}_t$ by (\ref{eq 3.1.2}). The
Bismut superconnection $\wh{\bbb}^{E, W_1}$ defining the Bismut--Cheeger eta form
$\wh{\eta}^{E, W_1}(g^E_1, \nabla^E_1, T^H_1X, g^{T^VX}_1, L_1\oplus W_1, \nabla^{L_1}
\oplus\nabla^{W_1})$ is given by
$$\wh{\bbb}^{E, W_1}=\wh{\bbb}^E\oplus\aaa^{W_1},$$
It follows from Lemma \ref{lemma 3.1.1} that (\ref{eq 3.1.25}) becomes
\begin{equation}\label{eq 3.1.26}
\begin{split}
&i_{B, 1}^*\wh{\eta}^{\e, \HH}(g^{\e}, \nabla^{\e}, T^H\wt{X}, g^{T^V\wt{X}},
\LLL\oplus\HH, \wt{\nabla}^{\LLL\oplus\HH})\\
=~&\wh{\eta}^E(g^E_1, \nabla^E_1, T^H_1X, g^{T^VX}_1, L_1)
\end{split}
\end{equation}
in \dis{\frac{\Omega^{\odd}(B)}{\im(d)}}. By considering the split quadruple
$(W_0, g^{W_0}, \nabla^{W_0}, s_{W_0})$ defined in a similar way as above we have
\begin{equation}\label{eq 3.1.27}
\begin{split}
&i_{B, 0}^*\wh{\eta}^{\e, \HH}(g^{\e}, \nabla^{\e}, T^H\wt{X}, g^{T^V\wt{X}},
\LLL\oplus\HH, \wt{\nabla}^{\LLL\oplus\HH})\\
=~&\wh{\eta}^E(g^E_0, \nabla^E_0, T^H_0X, g^{T^VX}_0, L_0)
\end{split}
\end{equation}
in \dis{\frac{\Omega^{\odd}(B)}{\im(d)}}. Thus by putting (\ref{eq 3.1.26}) and
(\ref{eq 3.1.27}) into (\ref{eq 3.1.24}) we have
\begin{displaymath}
\begin{split}
&\wh{\eta}^E(g^E_1, \nabla^E_1, T^H_1X, g^{T^VX}_1, L_1)-\wh{\eta}^E(g^E_0,
\nabla^E_0, T^H_0X, g^{T^VX}_0, L_0)\\
=&\int_{X/B}\wt{\wh{A}}(\nabla^{T^VX}_0, \nabla^{T^VX}_1)\wedge\ch(\nabla^E_0)+
\int_{X/B}\wh{A}(\nabla^{T^VX}_1)\wedge\CS(\nabla^E_0, \nabla^E_1)\\
&\qquad-\CS(h^*(\nabla^{L_0}\oplus\nabla^{W_0}), \nabla^{L_1}\oplus\nabla^{W_1})
\end{split}
\end{displaymath}
in \dis{\frac{\Omega^{\odd}(B)}{\im(d)}}. Thus (\ref{eq 3.1.5}) holds.
\end{proof}

We now give another proofs of \cite[(3) and (4) of Corollary 7.36]{FL10}.
\begin{coro}\label{coro 3.1.1}
Let $\pi:X\to B$ be a submersion with closed, oriented and spin fibers of even
dimension, and $\E=(E, g^E, \nabla^E, \omega)$ a generator of $\wh{K}_{\FL}(X)$. If
$L_0\to B$ and $L_1\to B$ are $\Z_2$-graded complex vector bundles satisfying the MF
property with respect to $\DD^{S\otimes E}$, then
\begin{equation}\label{eq 3.1.28}
\ind^{\an}_{\FL}(\E; L_0)=\ind^{\an}_{\FL}(\E; L_1).
\end{equation}
\end{coro}
\begin{proof}
By (\ref{eq 2.5.4}), (\ref{eq 2.5.5}) and (\ref{eq 2.5.7}), proving (\ref{eq 3.1.28})
is equivalent to showing the existence of $\Z_2$-graded generators $\W_0$ and $\W_1$ of
$\wh{K}_{\FL}(B)$ of the form (\ref{eq 2.5.3}) such that
\begin{equation}\label{eq 3.1.29}
L_0\oplus W_0\cong L_1\oplus W_1
\end{equation}
as $\Z_2$-graded complex vector bundles and
\begin{equation}\label{eq 3.1.30}
\begin{split}
&\wh{\eta}^E(g^E, \nabla^E, T^HX, g^{T^VX}, L_1)-\wh{\eta}^E(g^E, \nabla^E, T^HX,
g^{T^VX}, L_0)\\
=~&\CS(\nabla^{L_1}\oplus\nabla^{W_1}, \nabla^{L_0}\oplus\nabla^{W_0})
\end{split}
\end{equation}
in \dis{\frac{\Omega^{\odd}(B)}{\im(d)}}. First note that (\ref{eq 3.1.29}) follows from
(\ref{eq 3.1.15}). By taking $g^E_1=g^E_0=g^E$, $\nabla^E_1=\nabla^E_0=\nabla^E$,
$T^H_1X=T^H_2X=T^HX$ and $g^{T^VX}_1=g^{T^VX}_0=g^{T^VX}$ in (\ref{eq 3.1.5}), we obtain
(\ref{eq 3.1.30}) by (\ref{eq 2.1.8}).
\end{proof}
\begin{coro}\label{coro 3.1.2}
Let $\pi:X\to B$ be a submersion with closed, oriented and spin fibers of even dimension,
and $\E=(E, g^E, \nabla^E, \omega)$ a generator of $\wh{K}_{\FL}(X)$. If Assumption
\ref{ass 1} is satisfied, then for any $\Z_2$-graded complex vector bundle $L\to B$
satisfying the MF property with respect to $\DD^{S\otimes E}$ we have
\begin{equation}\label{eq 3.1.31}
\ind^{\an}_{\FL}(\E; L)=\ind^{\an}_{\FL}(\E),
\end{equation}
where the right-hand side of (\ref{eq 3.1.31}) is given by (\ref{eq 2.5.6}).
\end{coro}
\begin{proof}
Since the kernel bundle $\ker(\DD^{S\otimes E})\to B$ exists, and it satisfies the MF
property with respect to $\DD^{S\otimes E}$, it follows from Corollary \ref{coro 3.1.1}
that $\ind^{\an}_{\FL}(\E; L)=\ind^{\an}_{\FL}(\E; \ker(\DD^{S\otimes E}))$. Thus proving
(\ref{eq 3.1.31}) is equivalent to proving
\begin{equation}\label{eq 3.1.32}
\ind^{\an}_{\FL}(\E; \ker(\DD^{S\otimes E}))=\ind^{\an}_{\FL}(\E).
\end{equation}
As in the proof of Corollary \ref{coro 3.1.2}, proving (\ref{eq 3.1.32}) is equivalent to
showing the existence of $\Z_2$-graded generators $\W_0$ and $\W_1$ of $\wh{K}_{\FL}(B)$
of the form (\ref{eq 2.5.3}) such that
\begin{equation}\label{eq 3.1.33}
\ker(\DD^{S\otimes E})\oplus W_1\cong\ker(\DD^{S\otimes E})\oplus W_0
\end{equation}
as $\Z_2$-graded complex vector bundles and
\begin{equation}\label{eq 3.1.34}
\begin{split}
&\wh{\eta}^E(g^E, \nabla^E, T^HX, g^{T^VX}, \ker(\DD^{S\otimes E}))-\wt{\eta}^E
(g^E, \nabla^E, T^HX, g^{T^VX})\\
=&\CS(\nabla^{\ker(\DD^{S\otimes E})}\oplus\nabla^{W_1}, \nabla^{\ker(\DD^{S\otimes E})}
\oplus\nabla^{W_0})
\end{split}
\end{equation}
in \dis{\frac{\Omega^{\odd}(B)}{\im(d)}}. We choose $W_0\to B$ and $W_1\to B$ to be zero
$\Z_2$-graded complex vector bundles, so that (\ref{eq 3.1.33}) is satisfied. Since
$$\CS(\nabla^{\ker(\DD^{S\otimes E})}\oplus\nabla^{W_1}, \nabla^{\ker(\DD^{S\otimes E})}
\oplus\nabla^{W_0})=\CS(\nabla^{\ker(\DD^{S\otimes E})}, \nabla^{\ker(\DD^{S\otimes E})})
=0$$
in \dis{\frac{\Omega^{\odd}(B)}{\im(d)}}, it follows that (\ref{eq 3.1.34}) becomes
\begin{equation}\label{eq 3.1.35}
\wh{\eta}^E(g^E, \nabla^E, T^HX, g^{T^VX}, \ker(\DD^{S\otimes E}))=\wt{\eta}^E
(g^E, \nabla^E, T^HX, g^{T^VX})
\end{equation}
in \dis{\frac{\Omega^{\odd}(B)}{\im(d)}}. First note that in the current situation,
the $\Z_2$-graded Hermitian bundle $L^{\op}\to B$ in the definition of $\wh{\bbb}^E$
given by (\ref{eq 2.4.9}) is $\ker(\DD^{S\otimes E})^{\op}\to B$, and the unitary
connection $\nabla^{L, \op}$ is $\nabla^{\ker(\DD^{S\otimes E}), \op}$. To show
(\ref{eq 3.1.35}), fix any $T\geq 1$ and $t<a$. Since $\wh{\bbb}^E_t=\bbb^E_t\oplus
\nabla^{\ker(\DD^{S\otimes E}), \op}$ by Remark \ref{remark 4}, it follows from
(\ref{eq 2.1.13}), (\ref{eq 2.1.11}) and (\ref{eq 2.1.12}) that
\begin{displaymath}
\begin{split}
&\CS(\wh{\bbb}^E_T, \wh{\bbb}^E_t)-\CS(\bbb^E_T, \bbb^E_t)\\
=&\CS(\wh{\bbb}^E_T, \wh{\bbb}^E_t)-\CS(\bbb^E_T, \bbb^E_t)-\CS(\nabla^{\ker
(\DD^{S\otimes E}), \op}, \nabla^{\ker(\DD^{S\otimes E}), \op})\\
=&\CS(\wh{\bbb}^E_T, \wh{\bbb}^E_t)-\CS(\bbb^E_T\oplus\nabla^{\ker(\DD^{S\otimes E}),
\op}, \bbb^E_t\oplus\nabla^{\ker(\DD^{S\otimes E}), \op})\\
=&\CS(\wh{\bbb}^E_T, \wh{\bbb}^E_t)+\CS(\bbb^E_t\oplus\nabla^{\ker(\DD^{S\otimes E}),
\op}, \bbb^E_T\oplus\nabla^{\ker(\DD^{S\otimes E}), \op})\\
=&\CS(\wh{\bbb}^E_T, \bbb^E_T\oplus\nabla^{\ker(\DD^{S\otimes E}), \op})
\end{split}
\end{displaymath}
in \dis{\frac{\Omega^{\odd}(B)}{\im(d)}}. By letting $T\to\infty$ and $t\to 0$ in above,
we have
\begin{displaymath}
\begin{split}
&\wh{\eta}^E(g^E, \nabla^E, T^HX, g^{T^VX}, \ker(\DD^{S\otimes E}))-\wt{\eta}^E
(g^E, \nabla^E, T^HX, g^{T^VX})\\
=&\lim_{T\to\infty}\CS(\wh{\bbb}^E_T, \bbb^E_T\oplus\nabla^{\ker(\DD^{S\otimes E}),
\op})
\end{split}
\end{displaymath}
in \dis{\frac{\Omega^{\odd}(B)}{\im(d)}}. By the estimates in \cite[\S9.3]{BGV}
\footnote{This argument is inspired by the proof of \cite[Proposition 9]{GL15}.} we
have
$$\lim_{T\to\infty}\CS(\wh{\bbb}^E_T, \bbb^E_T\oplus\nabla^{\ker(\DD^{S\otimes E}),
\op})=0.$$
Thus (\ref{eq 3.1.35}) holds, and therefore so does (\ref{eq 3.1.32}).
\end{proof}

\subsection{Some properties of Cheeger--Chern--Simons class}\label{s 3.2}

In this subsection we prove some properties of the Cheeger--Chern--Simons class that are
needed to prove Theorem \ref{thm 3.3.1}. These properties roughly say that the
Cheeger--Chern--Simons classes of a complex flat vector bundle with flat connection with
respect to different gradings are equal.

\begin{lemma}\label{lemma 3.2.1}
Let $(E, v, \nabla^E)$ be a flat $\Z$-graded cochain complex of the form
(\ref{eq 2.2.8}). If $\rk(E^+)=\rk(E^-)$, then
\begin{equation}\label{eq 3.2.1}
\sum_{k=0}^m(-1)^k\CCS(E^k, \nabla^k)=\CCS(E, \nabla^E),
\end{equation}
where the left-hand side of (\ref{eq 3.2.1}) is the Cheeger--Chern--Simons class of
$(E, \nabla^E)$ with respect to the $\Z$ grading given by (\ref{eq 2.2.9}) and the
right-hand side of (\ref{eq 3.2.1}) is the Cheeger--Chern--Simons class of $(E,
\nabla^E)$ with respect to the $\Z_2$ grading given by (\ref{eq 2.2.10}).
\end{lemma}
\begin{proof}
Consider $E\to X$ as a $\Z$-graded complex flat vector bundle with $\Z$-graded flat
connection $\nabla^E$. Since $(\nabla^k)^2=0$ for each $0\leq k\leq m$, there exists
$\ell_k\in\N$ such that $\ell_kE^k\cong\ell_k\C^{r_k}$, where $r_k=\rk(E^k)$. By
(\ref{eq 2.2.1}) a differential form representative of $\CCS(E^k, \nabla^k)$ is given
by
$$\frac{1}{\ell_k}\CS(\nabla^{\ell_kE^k}_0, \ell_k\nabla^k),$$
where $\nabla^{\ell_kE^k}_0$ is a trivial connection on $\ell_kE^k\to X$. Let $\ell$
be the least common multiple of $\ell_0, \ldots, \ell_m$. Then there exist unique $d_0,
\ldots, d_m\in\N$ such that $\ell=d_k\ell_k$ for each $0\leq k\leq m$. Since
\begin{equation}\label{eq 3.2.2}
\ell E^k=d_k\ell_kE^k\cong d_k\ell_k\C^{r_k}=\ell\C^{r_k},
\end{equation}
it follows from (\ref{eq 2.1.9}) and (\ref{eq 2.1.8}) that
\begin{equation}\label{eq 3.2.3}
\begin{split}
\frac{1}{\ell_k}\CS(\nabla^{\ell_kE^k}_0, \ell_k\nabla^k)&=\frac{1}{\ell}\CS(d_k
\nabla^{\ell_kE^k}_0, d_k\ell_k\nabla^k)\\
&=\frac{1}{\ell}\CS(d_k\nabla^{\ell_kE^k}_0, \nabla^{\ell E^k}_0)+\frac{1}{\ell}\CS
(\nabla^{\ell E^k}_0, \ell\nabla^k),
\end{split}
\end{equation}
where $\nabla^{\ell E^k}_0$ is a trivial connection on $\ell E^k\to X$. By Remark
\ref{remark 2} we have
\begin{equation}\label{eq 3.2.4}
\CS(d_k\nabla^{\ell_kE^k}_0, \nabla^{\ell E^k}_0)=\ch^{\odd}(g_k)\in\Omega^{\odd}_\Q
(X; \C)
\end{equation}
for some smooth map $g_k:X\to\GL(\ell r_k; \C)$. By (\ref{eq 3.2.3}) and (\ref{eq 3.2.4})
we have
\begin{equation}\label{eq 3.2.5}
\sum_{k=0}^m\frac{(-1)^k}{\ell_k}\CS(\nabla^{\ell_kE^k}_0, \ell_k\nabla^k)=\sum_{k=0}^m
\frac{(-1)^k}{\ell}\CS(\nabla^{\ell E^k}_0, \ell\nabla^k)+\sum_{k=0}^m\frac{(-1)^k}{\ell}
\ch^{\odd}(g_k).
\end{equation}
Recall that the connections $\nabla^\pm$ on $E^\pm\to X$ are defined by (\ref{eq 2.2.8}).
Note that
$$\nabla^{\ell E^+}_0:=\bigoplus_k\nabla^{\ell E^{2k}}_0,\qquad\nabla^{\ell E^-}_0:=
\bigoplus_k\nabla^{\ell E^{2k+1}}_0$$
are trivial connections on $\ell E^\pm\to X$ respectively. By (\ref{eq 2.1.9}) the first
term of the right-hand side of (\ref{eq 3.2.5}) becomes
\begin{equation}\label{eq 3.2.6}
\sum_{k=0}^m\frac{(-1)^k}{\ell}\CS(\nabla^{\ell E^k}_0, \ell\nabla^k)=\frac{1}{\ell}\CS
(\nabla^{\ell E^+}_0, \ell\nabla^+)-\frac{1}{\ell}\CS(\nabla^{\ell E^-}_0, \ell\nabla^-).
\end{equation}
Since $\rk(E^+)=\rk(E^-)$, (\ref{eq 3.2.2}) induces an isomorphism $j:\ell E^+\to\ell E^-$.
Since $j^*\nabla^{\ell E^-}_0=\nabla^{\ell E^+}_0$, by (\ref{eq 2.1.8}) we have
\begin{equation}\label{eq 3.2.7}
\begin{split}
&\frac{1}{\ell}\CS(\nabla^{\ell E^+}_0, \ell\nabla^+)\\
=&\frac{1}{\ell}\CS(\nabla^{\ell E^+}_0, j^*\ell\nabla^-)+\frac{1}{\ell}\CS(j^*\ell\nabla^-,
\ell\nabla^+)\\
=&\frac{1}{\ell}\CS(\nabla^{\ell E^+}_0, j^*\ell\nabla^-)+\frac{1}{\ell}\CS(j^*\ell\nabla^-,
j^*\nabla^{\ell E^-}_0)+\frac{1}{\ell}\CS(j^*\nabla^{\ell E^-}_0, \nabla^{\ell E^+}_0)\\
=&\frac{1}{\ell}\CS(\nabla^{\ell E^+}_0, j^*\ell\nabla^-)+\frac{1}{\ell}\CS(j^*\ell\nabla^-,
j^*\nabla^{\ell E^-}_0).
\end{split}
\end{equation}
By (\ref{eq 3.2.6}) and (\ref{eq 3.2.7}), (\ref{eq 3.2.5}) becomes
\begin{equation}\label{eq 3.2.8}
\sum_{k=0}^m\frac{(-1)^k}{\ell_k}\CS(\nabla^{\ell_kE^k}_0, \ell_k\nabla^k)=\frac{1}{\ell}
\CS(j^*\ell\nabla^-, \ell\nabla^+)+\sum_{k=0}^m\frac{(-1)^k}{\ell}\ch^{\odd}(g_k).
\end{equation}
Since the mod $\Q$ reduction of the de Rham class of the left-hand side of
(\ref{eq 3.2.8}) is the left-hand side of (\ref{eq 3.2.1}), and the same is true for
the right-hand side of (\ref{eq 3.2.8}) and (\ref{eq 3.2.1}), it follows that
(\ref{eq 3.2.1}) holds.
\end{proof}

The following lemma follows from (\ref{eq 2.1.10}) and (\ref{eq 2.1.14}).
\begin{lemma}\label{lemma 3.2.3}
Let $F^\pm\to X$ be a complex flat vector bundle with flat connection $\nabla^\pm$.
Define $F:=F^+\oplus F^-$ and $\nabla^F:=\nabla^+\oplus\nabla^-$. If $F\to X$ and
$\nabla^F$ are ungraded direct sums, then
\begin{equation}\label{eq 3.2.9}
\CCS(F, \nabla^F)=\CCS(F^+, \nabla^+)+\CCS(F^-, \nabla^-).
\end{equation}
If $\rk(F^+)=\rk(F^-)$ and $F\to X$, $\nabla^F$ are $\Z_2$-graded direct sums, then
\begin{equation}\label{eq 3.2.10}
\CCS(F, \nabla^F)=\CCS(F^+, \nabla^+)-\CCS(F^-, \nabla^-).
\end{equation}
\end{lemma}

Note that (\ref{eq 3.2.10}) is not a consequence of Lemma \ref{lemma 3.2.1} since the
$\Z_2$ grading of $(F, \nabla^F)$ does not come from a $\Z$-graded flat cochain complex.

\subsection{The real part of the Riemann--Roch--Grothendieck theorem for complex flat
vector bundles}\label{s 3.3}

In this subsection we first prove a $\Z_2$-graded version of (\ref{eq 1.0.2}) for $\dim(Z)$
even at the differential form level (Theorem \ref{thm 3.3.1}). Then we apply Theorem
\ref{thm 3.3.1} and a result by Bismut \cite[Theorem 3.12]{B05} to deduce (\ref{eq 1.0.2})
for $\dim(Z)$ even.

Let $n\in\N$. In the proof of Theorem \ref{thm 3.3.1} we will use the following notations.
Write $g^n$ for a trivial metric on the trivial bundle $\C^n\to B$. Let $d^n$ be a trivial
unitary connection on $\C^n\to B$. Write $\un{\C^n}\to B$ for the $\Z_2$-graded trivial
bundle $\un{\C^n}:=\C^n\oplus\C^n$ and $\un{g^n}$ for the $\Z_2$-graded trivial metric
defined by $\un{g^n}=g^n\oplus g^n$. Let $\un{d^n}$ be a $\Z_2$-graded trivial unitary
connection on $\un{\C^n}\to B$ defined by $\un{d^n}=d^n\oplus d^n$.
\begin{thm}\label{thm 3.3.1}
Let $\pi:X\to B$ be a submersion with closed fibers $Z$ with $\dim(Z)$ even and $F\to X$
a $\Z_2$-graded complex flat vector bundle of virtual rank zero with $\Z_2$-graded flat
connection $\nabla^F=\nabla^+\oplus\nabla^-$. Put a $\Z_2$-graded Hermitian metric
$g^F=g^+\oplus g^-$ on $F\to X$, and the induced Hermitian metric $g^{H(Z, F^\pm|_Z)}$
on $H(Z, F^\pm|_Z)\to B$. Define a unitary connection $\nabla^{\pm, u}$ on $F^\pm\to X$
with respect to $g^\pm$ by (\ref{eq 2.2.3}) and a $\Z_2$-graded unitary connection
$\nabla^{H(Z, F^\pm|_Z), u}$ on $H(Z, F^\pm|_Z)\to B$ with respect to
$g^{H(Z, F^\pm|_Z)}$ by (\ref{eq 2.3.7}) respectively. Then there exist $k, N, M\in\N$,
a smooth isometric isomorphism $\wt{j}:kF^+\to kF^-$, and a smooth $\Z_2$-graded
isometric isomorphism $\wt{f}:kH(Z, F^+|_Z)\oplus\un{\C^M}\to kH(Z, F^-|_Z)\oplus
\un{\C^N}$ such that
\begin{equation}\label{eq 3.3.1}
\CS(\wt{f}^*(k\nabla^{H(Z, F^-|_Z), u}\oplus\un{d^N}), k\nabla^{H(Z, F^+|_Z), u}\oplus
\un{d^M})=\int_{X/B}e(\nabla^{T^VX})\wedge\CS(\wt{j}^*k\nabla^{-, u}, k\nabla^{+, u})
\end{equation}
in \dis{\frac{\Omega^{\odd}(B)}{\Omega^{\odd}_\Q(B)}}.
\end{thm}
\begin{proof}
Let $F\to X$ be a $\Z_2$-graded complex flat vector bundle of virtual rank zero with
$\Z_2$-graded flat connection $\nabla^F=\nabla^+\oplus\nabla^-$. Then there exist
$k_1\in\N$ and a smooth bundle isomorphism $j:k_1F^+\to k_1F^-$.

Let $\bullet\in\set{\even, \odd}$. Define $n^\bullet_\pm=\rk(H^\bullet(Z, F^\pm|_Z))$.
Since $H^\bullet(Z, F^\pm|_Z)\to B$ is a complex flat vector bundle with flat connection
$\nabla^{H^\bullet(Z, F^\pm|_Z)}$, there exist $\ell^\bullet_\pm\in\N$ and smooth bundle
isomorphisms $h^\bullet_\pm:\ell^\bullet_\pm H^\bullet(Z, F^\pm|_Z)\to\ell^\bullet_\pm
\C^{n^\bullet_\pm}$. Let $k\in\N$ be the least common multiple of
\begin{equation}\label{eq 3.3.2}
k_1,\quad\ell^{\even}_+,\quad\ell^{\even}_-,\quad\ell^{\odd}_+\textrm{ and }\ell^{\odd}_-.
\end{equation}
We still denote by
\begin{equation}\label{eq 3.3.3}
j:kF^+\to kF^-\textrm{ and }h^\bullet_\pm:kH^\bullet(Z, F^\pm|_Z)\to k\C^{n^\bullet_\pm}
\end{equation}
the resulting smooth bundle isomorphisms. Note that our choice of $k$ guarantees that
smooth bundle isomorphisms in (\ref{eq 3.3.3}) hold simultaneously.

Put a $\Z_2$-graded Hermitian metric $g^F=g^+\oplus g^-$ on $F\to X$. Define a unitary
connection $\nabla^{\pm, u}$ on $F^\pm\to X$ by (\ref{eq 2.2.3}) with respect to $g^\pm$.
By (\ref{eq 2.2.11}) there exists $f\in\Aut(kF^+)$ such that $kg^+=f^*j^*kg^-$. Write
$\wt{j}=j\circ f$. Thus $\wt{j}:kF^+\to kF^-$ is a smooth isometric isomorphism. On the
other hand, since $\wt{j}^*k\nabla^-$ is a flat connection on $kF^+\to X$, it follows that
\begin{equation}\label{eq 3.3.4}
\wt{j}:(kF^+, \wt{j}^*k\nabla^-)\to(kF^-, k\nabla^-)
\end{equation}
is an isomorphism of complex flat vector bundles.\footnote{We will use the different flat
structures $(kF^+, k\nabla^+)$ and $(kF^+, \wt{j}^*k\nabla^-)$ separately.} Write $\ff=
kF^+$ and $g^{\ff}=kg^+$.  Let $\nabla^{\ff}_t$ be a smooth curve of connections on
$\ff\to X$ such that
$$\nabla^{\ff}_1=k\nabla^+,\qquad\nabla^{\ff}_0=\wt{j}^*k\nabla^-.$$
Consider the pullback of $\ff\to X$ by $p_X$:
\cdd{\f @>>> \ff \\ @VVV @VVV \\ \wt{X} @>>p_X> X}
where $\f:=p_X^*\ff$. Define a connection $\nabla^{\f}$ on $\f\to\wt{X}$ by
$$\nabla^{\f}=\nabla^{\ff}_t+dt\wedge\frac{\partial}{\partial t}.$$
Note that $\nabla^{\f}$ is not necessarily flat. Define a Hermitian metric on $\f\to\wt{X}$
by $g^{\f}=p_X^*g^{\ff}$, and the unitary connection $\nabla^{\f, u}$ on $\f\to\wt{X}$ by
(\ref{eq 2.2.3}). Note that $\wt{j}^*k\nabla^{-, u}$ is a unitary connection on $\ff\to X$
with respect to $g^{\ff}$. Since
$$\nabla^{\ff, u}_1=\nabla^{\ff}_1+\frac{1}{2}(g^{\ff})^{-1}(\nabla^{\ff}_1g^{\ff})=
k\nabla^++\frac{1}{2}(kg^+)^{-1}(k\nabla^+(kg^+))=k\nabla^{+, u}$$
and
\begin{displaymath}
\begin{split}
\nabla^{\ff, u}_0&=\nabla^{\ff}_0+\frac{1}{2}(g^{\ff})^{-1}(\nabla^{\ff}_0g^{\ff})\\
&=\wt{j}^*k\nabla^-+\frac{1}{2}(kg^+)^{-1}(\wt{j}^*k\nabla^-(kg^+))\\
&=\wt{j}^*\bigg(k\nabla^-+\frac{1}{2}(kg^-)^{-1}(k\nabla^-(kg^-))\bigg)\\
&=\wt{j}^*k\nabla^{-, u},
\end{split}
\end{displaymath}
it follows that
$$i_{X, 1}^*\nabla^{\f, u}=\nabla^{\ff, u}_1=k\nabla^{+, u},\qquad i_{X, 0}^*
\nabla^{\f, u}=\nabla^{\ff, u}_0=\wt{j}^*k\nabla^{-, u}.$$

Temporarily assume that the fibers $Z$ are oriented and spin. Then the fibers of
$\wt{\pi}:\wt{X}\to\wt{B}$, denoted by $\wt{Z}$, are also oriented and spin. The geometric
data on $\wt{\pi}:\wt{X}\to\wt{B}$ is given by $(g^{\f}, \nabla^{\f, u}, T^H\wt{X},
g^{T^V\wt{X}})$, where $T^H\wt{X}\to\wt{X}$ and $g^{T^V\wt{X}}$ are obtained by pulling
back a fixed choice of $T^HX\to X$ and $g^{T^VX}$ respectively. The infinite rank
$\Z_2$-graded complex vector bundle $\wt{\pi}_*^{\spin}(S(T^V\wt{X})^*\otimes\f)\to\wt{B}$
defined by (\ref{eq 2.4.1}) is equipped with a $\Z_2$-graded $L^2$-metric and a
$\Z_2$-graded unitary connection $\nabla^{\wt{\pi}_*^{\spin}(S(T^V\wt{X})^*\otimes\f), u}$
defined by (\ref{eq 2.4.2}). Define the spin Dirac operator $\DD^{S\wh{\otimes}(S^*\otimes
\f)}$ twisted by $S(T^V\wt{X})^*\otimes\f\to\wt{X}$ by (\ref{eq 2.4.0}), and define
$\wt{V}$ by (\ref{eq 2.3.9}). The perturbed twisted spin Dirac operator $\DD^{S\wh{\otimes}
(S^*\otimes\f)}+\wt{V}$ acting on
$$\Gamma(\wt{X}, S(T^V\wt{X})\wh{\otimes}(S(T^V\wt{X})^*\otimes\f))$$
can be considered as an odd self-adjoint element in
$$\Gamma(\wt{B}, \End^-(\wt{\pi}_*^{\spin}(S(T^V\wt{X})^*\otimes\f))).$$
We do not assume $\DD^{S\wh{\otimes}(S^*\otimes\f)}+\wt{V}$ satisfies Assumption
\ref{ass 1}. Let $\LLL\to\wt{B}$ be a $\Z_2$-graded complex vector bundle satisfying the
MF property with respect to $\DD^{S\wh{\otimes}(S^*\otimes\f)}+\wt{V}$. As in the proof
of Proposition \ref{prop 3.1.1}, write $L\to B$ for $i_{B, 1}^*\LLL\to B$, and note that
$i_{B, 1}^*\LLL\cong i_{B, 0}^*\LLL$.

Note that
\begin{equation}\label{eq 3.3.5}
S(T^VX)\wh{\otimes}S(T^VX)^*\cong\Lambda(T^VX)^*\otimes\C
\end{equation}
as $\Z_2$-graded complex vector bundles. By suppressing the notation $\otimes\C$ in
(\ref{eq 3.3.5}) we have
\begin{equation}\label{eq 3.3.6}
\begin{split}
S(T^VX)\wh{\otimes}(S(T^VX)^*\otimes\ff)&\cong(S(T^VX)\wh{\otimes}S(T^VX)^*)\otimes\ff\\
&\cong\Lambda(T^VX)^*\otimes\ff
\end{split}
\end{equation}
as $\Z_2$-graded complex vector bundles. By (\ref{eq 2.3.2}), (\ref{eq 2.4.1}) and
(\ref{eq 3.3.6}) we have $\pi^\Lambda_*\ff\cong\pi^{\spin}_*(S(T^VX)^*\otimes\ff)$, and
therefore
$$\Gamma(B, \pi^\Lambda_*\ff)\cong\Gamma(B, \pi^{\spin}_*(S(T^VX)^*\otimes\ff)).$$
Moreover, under the isomorphism (\ref{eq 3.3.5}) the Clifford multiplication on the
Clifford module $S(T^VX)\to X$ corresponds to the Clifford multiplication on the
Clifford module on $\Lambda(T^VX)^*\to X$ \cite[Proposition 3.5]{BGV}, and
\begin{equation}\label{eq 3.3.7}
\nabla^{S(T^VX)\wh{\otimes}S(T^VX)^*}=\nabla^{\Lambda(T^VX)^*}
\end{equation}
as $\Z_2$-graded unitary connections, where $\nabla^{S(T^VX)\wh{\otimes}S(T^VX)^*}$ is the
$\Z_2$-graded tensor product of $\nabla^{S(T^VX)}$ and $\nabla^{S(T^VX)^*}$.

Recall that $\DD^{S\wh{\otimes}(S^*\otimes\ff)}_1+V_1$ is defined in terms of, among other
things, $(\ff, \nabla^{\ff, u}_1)$, where $\nabla^{\ff, u}_1=k\nabla^{+, u}$. By
(\ref{eq 2.3.1}), (\ref{eq 2.4.1}) and (\ref{eq 3.3.7}) we have
\begin{equation}\label{eq 3.3.8}
\DD^{\Lambda\otimes\ff}_1=\DD^{S\wh{\otimes}(S^*\otimes\ff)}_1.
\end{equation}
Thus by (\ref{eq 2.3.8}) and (\ref{eq 3.3.8}) we have
\begin{equation}\label{eq 3.3.9}
\DD^{Z, \dr}_1=\DD^{\Lambda\otimes\ff}_1+V_1=\DD^{S\wh{\otimes}(S^*\otimes\ff)}_1+V_1.
\end{equation}
Since the family of kernels of $\DD_1^{Z, \dr}$ form a $\Z_2$-graded complex vector bundle,
the same is true for $\DD^{S\wh{\otimes}(S^*\otimes\ff)}_1+V_1$. Denoted by
$\ker(\DD^{S\wh{\otimes}(S^*\otimes\ff)}_1+V_1)\to B$ the resulting $\Z_2$-graded complex
vector bundle. Similarly, $\DD^{S\wh{\otimes}(S^*\otimes\ff)}_0+V_0$ is defined in terms
of, among other things, $(\ff, \nabla^{\ff, u}_0)$, where $\nabla^{\ff, u}_0=\wt{j}^*k
\nabla^{-, u}$. By a similar argument the family of the kernels of
$\DD^{S\wh{\otimes}(S^*\otimes\ff)}_0+V_0$ form a $\Z_2$-graded complex vector bundle,
denoted by $\ker(\DD^{S\wh{\otimes}(S^*\otimes\ff)}_0+V_0)\to B$.

Let $i\in\set{0, 1}$. Since $\ker(\DD^{S\wh{\otimes}(S^*\otimes\ff)}_i+V_i)\to B$ satisfies
the MF property with respect to $\DD^{S\wh{\otimes}(S^*\otimes\ff)}_i+V_i$, by
(\ref{eq 3.1.15}) there exist $\Z_2$-graded complex vector bundles $H\to B$ and $W_i\to B$
of the form $H^+=H^-$ and $W_i^+=W_i^-$ such that
\begin{equation}\label{eq 3.3.10}
\ker(\DD^{S\wh{\otimes}(S^*\otimes\ff)}_0+V_0)\oplus W_0\cong L\oplus H\cong\ker
(\DD^{S\wh{\otimes}(S^*\otimes\ff)}_1+V_1)\oplus W_1
\end{equation}
as $\Z_2$-graded complex vector bundles. By (\ref{eq 2.3.5}) and (\ref{eq 3.3.9}) we have
\begin{equation}\label{eq 3.3.11}
\ker(\DD^{S\wh{\otimes}(S^*\otimes\ff)}_1+V_1)\cong H(Z, \ff|_Z)=H(Z, kF^+|_Z)
\end{equation}
as $\Z_2$-graded complex vector bundles, where the flat connection $\nabla^{H(Z, kF^+|_Z)}$
on $H(Z, kF^+|_Z)\to B$ in (\ref{eq 3.3.11}) is defined in terms of $\nabla^{\ff}_1=
k\nabla^+$. By a similar argument we have
\begin{equation}\label{eq 3.3.12}
\ker(\DD^{S\wh{\otimes}(S^*\otimes\ff)}_0+V_0)\cong H(Z, \ff|_Z, \nabla^{\ff}_0)=H(Z, kF^+|_Z,
\wt{j}^*k\nabla^-)
\end{equation}
as $\Z_2$-graded complex vector bundles. In (\ref{eq 3.3.12}) the notation for the cohomology
bundle is to emphasize that it is defined by the corresponding flat connection. Since
$(kF^+, \wt{j}^*k\nabla^-)\cong(kF^-, k\nabla^-)$ as complex flat vector bundles by
(\ref{eq 3.3.4}), (\ref{eq 3.3.12}) becomes
\begin{equation}\label{eq 3.3.13}
\ker(\DD^{S\wh{\otimes}(S^*\otimes\ff)}_0+V_0)\cong H(Z, kF^+|_Z, \wt{j}^*k\nabla^-)\cong
H(Z, kF^-|_Z)
\end{equation}
as $\Z_2$-graded complex vector bundles. Since for any complex flat vector bundle $E\to X$
with flat connection $\nabla^E$ and any $k\in\N$,
$$(H(Z, kE|_Z), \nabla^{H(Z, kE|_Z)})\cong(kH(Z, E|_Z), k\nabla^{H(Z, E|_Z)})$$
as $\Z$-graded complex flat vector bundles, it follows from (\ref{eq 3.3.11}),
(\ref{eq 3.3.12}) and (\ref{eq 3.3.13}) that (\ref{eq 3.3.10}) becomes
\begin{equation}\label{eq 3.3.14}
kH(Z, F^-|_Z)\oplus W_0\cong L\oplus H\cong kH(Z, F^+|_Z)\oplus W_1.
\end{equation}
Consider the following bundle isomorphism of the even part of (\ref{eq 3.3.14}):
\begin{equation}\label{eq 3.3.15}
kH^{\even}(Z, F^-|_Z)\oplus W_0^+\cong kH^{\even}(Z, F^+|_Z)\oplus W_1^+.
\end{equation}
Since $B$ is assumed to be compact, there exists a complex vector bundle $V\to B$ such
that $W_1^+\oplus V\cong\C^m$ for some $m\in\N$. By direct summing $V\to B$ and $\ell$
copies of $\C^m\to B$ on both sides of (\ref{eq 3.3.15}), where $\ell\in\N$, it becomes
\begin{equation}\label{eq 3.3.16}
kH^{\even}(Z, F^-|_Z)\oplus\wt{W_0^+}\cong kH^{\even}(Z, F^+|_Z)\oplus\C^{(\ell+1)m},
\end{equation}
where $\wt{W_0^+}:=W_0^+\oplus V\oplus\C^{\ell m}$. By choosing and fixing a sufficiently
large $\ell$, we may assume that $2\rk(\wt{W_0^+})\geq\dim(B)$. Since $kH^{\even}
(Z, F^\pm|_Z)\cong k\C^{n^{\even}_\pm}\cong\C^{kn^{\even}_\pm}$ by (\ref{eq 3.3.3}),
(\ref{eq 3.3.16}) becomes
$$\C^{kn^{\even}_-}\oplus\wt{W_0^+}\cong\C^{kn^{\even}_++(\ell+1)m}.$$
Thus $kn^{\even}_-+\rk(\wt{W_0^+})=kn^{\even}_++(\ell+1)m$. Since $\C^{kn^{\even}_+
+(\ell+1)m}\cong\C^{kn^{\even}_-}\oplus\C^{\rk(\wt{W_0^+})}$, it follows that
$$\C^{kn^{\even}_-}\oplus\wt{W_0^+}\cong\C^{kn^{\even}_-}\oplus\C^{\rk(\wt{W_0^+})}.$$
By \cite[Theorem 1.5 of Chapter 9]{H94} (see also Remark \ref{remark 1}) we have
$\wt{W_0^+}\cong\C^{\rk(\wt{W_0^+})}$. Since $W_0^-=W_0^+$ and $W_1^-=W_1^+$, it follows
that
$$\wt{W_0^-}:=W_0^-\oplus V\oplus\C^{\ell m}=\wt{W_0^+}\cong\C^{\rk(\wt{W_0^+})}.$$
The above argument shows that there exist sufficiently large $N, M\in\N$ such that the
following diagram commutes
\begin{equation}\label{eq 3.3.17}
\xymatrix{ & L\oplus\wh{H} \ar[dl]_\cong \ar[dr]^h & \\ kH(Z, F^-|_Z)\oplus\un{\C^N} & &
kH(Z, F^+|_Z) \oplus\un{\C^M} \ar[ll]^{f=f_0\oplus f_1}_\cong}
\end{equation}
where $\wh{H}\to B$ is the resulting $\Z_2$-graded complex vector bundle, and $f$, $h$
are the resulting smooth $\Z_2$-graded bundle isomorphisms.

Put $\Z_2$-graded trivial metrics $\un{g^N}$ and $\un{g^M}$ and $\Z_2$-graded trivial
unitary connections $\un{d^N}$ and $\un{d^M}$ on $\un{\C^N}\to B$ and $\un{\C^M}\to B$
respectively. Since
$$kg^{H^{\even}(Z, F^+|_Z)}\oplus g^M,\quad f_0^*(kg^{H^{\even}(Z, F^-|_Z)}\oplus g^N)$$
are Hermitian metrics on $kH^{\even}(Z, F^+|_Z)\oplus\C^M\to B$, by (\ref{eq 2.2.11})
there exists a smooth isometric isomorphism
\begin{equation}\label{eq 3.3.18}
\wt{f_0}:kH^{\even}(Z, F^+|_Z)\oplus\C^M\to kH^{\even}(Z, F^-|_Z)\oplus\C^N.
\end{equation}
Similarly, denote by
\begin{equation}\label{eq 3.3.19}
\wt{f_1}:kH^{\odd}(Z, F^+|_Z)\oplus\C^M\to kH^{\odd}(Z, F^-|_Z)\oplus\C^N
\end{equation}
the corresponding smooth isometric isomorphism. Then
$$\wt{f}:=\wt{f_0}\oplus\wt{f_1}:kH(Z, F^+|_Z)\oplus\un{\C^M}\to kH(Z, F^-|_Z)\oplus
\un{\C^N}$$
is a smooth $\Z_2$-graded isometric isomorphism, and $\wt{f}^*(k\nabla^{H(Z, F^-|_Z), u}
\oplus\un{d^N})$ is a $\Z_2$-graded unitary connection on $kH(Z, F^+|_Z)\oplus\un{\C^M}
\to B$ with respect to $kg^{H(Z, F^+|_Z)}\oplus\un{g^M}$.

Note that $h^*(k\nabla^{H(Z, F^+|_Z), u}\oplus\un{d^M})$ and $h^*\wt{f}^*
(k\nabla^{H(Z, F^-|_Z), u}\oplus\un{d^N})$ are $\Z_2$-graded unitary connections on
$L\oplus\wh{H}\to B$ with respect to $h^*(kg^{H(Z, F^+|_Z)}\oplus\un{g^M})$. Define a
$\Z_2$-graded complex vector bundle $\HH\to\wt{B}$ by $\HH=p_B^*\wh{H}$, and a unitary
connection $\nabla^{\LLL\oplus\HH}$ on $\LLL\oplus\HH\to\wt{B}$ by (\ref{eq 2.1.6}) such
that
\begin{displaymath}
\begin{split}
i_{B, 1}^*\nabla^{\LLL\oplus\HH}&=h^*(k\nabla^{H(Z, F^+|_Z), u}\oplus\un{d^M}),\\
i_{B, 0}^*\nabla^{\LLL\oplus\HH}&=h^*\wt{f}^*(k\nabla^{H(Z, F^-|_Z), u}\oplus\un{d^N}).
\end{split}
\end{displaymath}

The term $\wh{A}(\nabla^{T^V\wt{X}})\wedge\ch(\nabla^{\e})$ in (\ref{eq 3.1.19}) becomes
$$\wh{A}(\nabla^{T^V\wt{X}})\wedge\ch(\nabla^{S(T^V\wt{X})^*\otimes\f, u}),$$
where $\nabla^{S(T^V\wt{X})^*\otimes\f, u}$ is the tensor product of
$\nabla^{S(T^V\wt{X})^*}$ and $\nabla^{\f, u}$. By \cite[(8.30)]{AB68} (see also
\cite[(3.46)]{B05} and \cite[Proposition 11.24 in p.328]{LM89}) we have
\begin{equation}\label{eq 3.3.20}
\begin{split}
&\wh{A}(\nabla^{T^V\wt{X}})\wedge\ch(\nabla^{S(T^V\wt{X})^*\otimes\f, u})\\
=&\wh{A}(\nabla^{T^V\wt{X}})\wedge\ch(\nabla^{S(T^V\wt{X})^*})\wedge\ch(\nabla^{\f, u})\\
=&\wh{A}(\nabla^{T^V\wt{X}})\wedge(\ch(\nabla^{S(T^V\wt{X})^*, +})-\ch(\nabla^{S(T^V
\wt{X})^*, -}))\wedge\ch(\nabla^{\f, u})\\
=&\wh{A}(\nabla^{T^V\wt{X}})\wedge\frac{e(\nabla^{T^V\wt{X}})}{\wh{A}(\nabla^{T^V\wt{X}})}
\wedge\ch(\nabla^{\f, u})\\
=&e(\nabla^{T^V\wt{X}})\wedge\ch(\nabla^{\f, u}).
\end{split}
\end{equation}
By (\ref{eq 3.1.5}) and (\ref{eq 3.3.20}) we have
\begin{equation}\label{eq 3.3.21}
\begin{split}
&\wh{\eta}^{S(T^VX)^*\otimes\ff}(kg^+, k\nabla^{+, u}, T^HX, g^{T^VX}, \ker(\DD^{S
\wh{\otimes}S^*\otimes\ff}_1+V_1))\\
&\qquad-\wh{\eta}^{S(T^VX)^*\otimes\ff}(kg^+, \wt{j}^*k\nabla^{-, u}, T^HX, g^{T^VX},
\ker(\DD^{S\wh{\otimes}S^*\otimes\ff}_0+V_0))\\
=&\int_{X/B}e(\nabla^{T^VX})\wedge\CS(\wt{j}^*k\nabla^{-, u}, k\nabla^{+, u})\\
&-\CS(h^*\wt{f}^*(k\nabla^{H(Z, F^-|_Z), u}\oplus\un{d^N}), h^*(k\nabla^{H(Z, F^+|_Z), u}
\oplus\un{d^M}))
\end{split}
\end{equation}
in \dis{\frac{\Omega^{\odd}(B)}{\im(d)}}. Since $h$ covers the identity map $\id_B$,
it follows from (\ref{eq 2.2.10}) that
\begin{equation}\label{eq 3.3.22}
\begin{split}
&\CS(h^*(\wt{f}^*k\nabla^{H(Z, F^-|_Z), u}\oplus\un{d^N}), h^*(k\nabla^{H(Z, F^+|_Z), u}
\oplus\un{d^M}))\\
=&\CS(\wt{f}^*(k\nabla^{H(Z, F^-|_Z), u}\oplus\un{d^N}), k\nabla^{H(Z, F^+|_Z), u}\oplus
\un{d^M})
\end{split}
\end{equation}
in \dis{\frac{\Omega^{\odd}(B)}{\im(d)}}. On the other hand, by (\ref{eq 3.1.35}) we
have
\begin{displaymath}
\begin{split}
&\wh{\eta}^{S(T^VX)^*\otimes\ff}(kg^+, k\nabla^{+, u}, T^HX, g^{T^VX}, \ker(\DD^{S
\wh{\otimes}S^*\otimes\ff}_1+V_1))\\
=&\wt{\eta}^{S(T^VX)^*\otimes\ff}(kg^+, k\nabla^{+, u}, T^HX, g^{T^VX})
\end{split}
\end{displaymath}
in \dis{\frac{\Omega^{\odd}(B)}{\im(d)}}. By (\ref{eq 3.3.7}) and (\ref{eq 3.3.6}) we
have
\begin{displaymath}
\begin{split}
&\nabla^{\Lambda(T^VX)^*}\otimes\id_{\Gamma(X, \ff)}+\id_{\Gamma(X, \Lambda(T^VX)^*)}
\otimes\nabla^{\ff, u}_1\\
=&\nabla^{S(T^VX)\wh{\otimes}S(T^VX)^*}\otimes\id_{\Gamma(X, \ff)}+\id_{\Gamma(X, S(T^VX)
\wh{\otimes}S(T^VX)^*)}\otimes\nabla^{\ff, u}_1\\
=&\nabla^{S(T^VX)}\otimes\id_{\Gamma(X, S(T^VX)^*\otimes\ff)}+\id_{\Gamma(X, S(T^VX))}
\otimes\nabla^{S(T^VX)^*\otimes\ff, u},
\end{split}
\end{displaymath}
where $\nabla^{S(T^VX)^*\otimes\ff, u}$ is the tensor product connection of
$\nabla^{S(T^VX)^*}$ and $\nabla^{\ff, u}_1=k\nabla^{+, u}$. It follows from
(\ref{eq 2.3.4}), (\ref{eq 2.4.2}) and above that
\begin{equation}\label{eq 3.3.23}
\nabla^{\pi^\Lambda_*\ff, u}=\nabla^{\pi^{\spin}_*(S(T^VX)^*\otimes\ff), u}.
\end{equation}
By (\ref{eq 3.3.8}) and (\ref{eq 3.3.23}) we see that $\bbb^{S(T^VX)^*\otimes\ff}=
\bbb^{\dr}$. By (\ref{eq 2.3.12}) we have
$$\wt{\eta}^{S(T^VX)^*\otimes\ff}(kg^+, k\nabla^{+, u}, T^HX, g^{T^VX})=\wt{\eta}^{\dr}
=0.$$
A similar argument shows that
\begin{displaymath}
\begin{split}
&\wh{\eta}^{S(T^VX)^*\otimes\ff}(kg^+, \wt{j}^*k\nabla^{-, u}, T^HX, g^{T^VX},
\ker(\DD^{S\wh{\otimes}S^*\otimes\ff}_0+V_0))\\
=&\wt{\eta}^{S(T^VX)^*\otimes\ff}(kg^+, \wt{j}^*k\nabla^{-, u}, T^HX, g^{T^VX})=
\wt{\eta}^{\dr}=0.
\end{split}
\end{displaymath}
By (\ref{eq 3.3.22}), (\ref{eq 3.3.21}) becomes
\begin{equation}\label{eq 3.3.24}
\CS(\wt{f}^*(k\nabla^{H(Z, F^-|_Z), u}\oplus\un{d^N}), k\nabla^{H(Z, F^+|_Z), u}\oplus
\un{d^M})=\int_{X/B}e(\nabla^{T^VX})\wedge\CS(\wt{j}^*k\nabla^{-, u}, k\nabla^{+, u})
\end{equation}
in \dis{\frac{\Omega^{\odd}(B)}{\im(d)}}. Since all the above computations are local, it
follows that (\ref{eq 3.3.24}) holds without assuming the fibers $Z$ are oriented and
spin. Thus (\ref{eq 3.3.1}) holds.
\end{proof}

The following corollary is a $\Z_2$-graded version of (\ref{eq 1.0.2}) for $\dim(Z)$ even.
\begin{coro}\label{coro 3.3.3}
Let $\pi:X\to B$ be a submersion with closed fibers $Z$ with $\dim(Z)$ even, and $F\to X$
a $\Z_2$-graded complex flat vector bundle of virtual rank zero with $\Z_2$-graded flat
connection $\nabla^{F}=\nabla^+\oplus\nabla^-$. By putting a $\Z_2$-graded Hermitian
metric $g^F$ on $F\to X$ and the induced $\Z$-graded Hermitian metric $g^{H(Z, F|_Z)}$
on $H(Z, F|_Z)\to B$ we have
\begin{equation}\label{eq 3.3.25}
\re(\CCS(H(Z, F|_Z), \nabla^{H(Z, F|_Z)}))=\int_{X/B}e(T^VX)\cup\re(\CCS(F, \nabla^{F}))
\end{equation}
in $H^{\odd}(B; \R/\Q)$.
\end{coro}

In the following proof we adopt the notations in the proof of Theorem \ref{thm 3.3.1}.
\begin{proof}
Let $F\to X$ be a $\Z_2$-graded complex flat vector bundle of virtual rank zero with
$\Z_2$-graded flat connection $\nabla^{F}=\nabla^+\oplus\nabla^-$. As in the proof of
Theorem \ref{thm 3.3.1} let $k\in\N$ be the least common multiple of the integers in
(\ref{eq 3.3.2}).

Since the $\Z_2$-grading of $H(Z, F|_Z)\to B$ is given by
\begin{displaymath}
\begin{split}
H(Z, F|_Z)^+&=H^{\even}(Z, F^+|_Z)\oplus H^{\odd}(Z, F^-|_Z),\\
H(Z, F|_Z)^-&=H^{\even}(Z, F^-|_Z)\oplus H^{\odd}(Z, F^+|_Z),
\end{split}
\end{displaymath}
it follows that
\begin{displaymath}
\begin{split}
&\rk(H(Z, F|_Z)^+)-\rk(H(Z, F|_Z)^-)\\
=&\rk(H^{\even}(Z, F^+|_Z))+\rk(H^{\odd}(Z, F^-|_Z))-\rk(H^{\even}(Z, F^-|_Z))\\
&\quad-\rk(H^{\odd}(Z, F^+|_Z))\\
=&\rk(H(Z, F^+|_Z))-\rk(H(Z, F^-|_Z))\\
=&\chi(Z)\rk(F^+)-\chi(Z)\rk(F^-)=0.
\end{split}
\end{displaymath}
Thus by Lemma \ref{lemma 3.2.3} we have
\begin{equation}\label{eq 3.3.26}
\begin{split}
&\re(\CCS(H(Z, F|_Z), \nabla^{H(Z, F|_Z)}))\\
=&\re(\CCS(H(Z, F|_Z)^+, \nabla^{H(Z, F|_Z)^+}))
-\re(\CCS(H(Z, F|_Z)^-, \nabla^{H(Z, F|_Z)^-}))\\
=&\re(\CCS(H^{\even}(Z, F^+|_Z), \nabla^{H^{\even}(Z, F^+|_Z)}))+\re(\CCS(H^{\odd}(Z,
F^-|_Z), \nabla^{H^{\odd}(Z, F^-|_Z)}))\\
&-\re(\CCS(H^{\even}(Z, F^-|_Z), \nabla^{H^{\even}(Z, F^-|_Z)}))-\re(\CCS(H^{\odd}(Z,
F^+|_Z), \nabla^{H^{\odd}(Z, F^+|_Z)})).
\end{split}
\end{equation}
Recall from (\ref{eq 3.3.3}) that the choice of $k$ guarantees the existence of the
smooth bundle isomorphisms $j:kF^+\to kF^-$ and $h^\bullet_\pm:kH^\bullet(Z, F^\pm)\to
\C^{kn^{\bullet}_\pm}$. Put a $\Z_2$-graded Hermitian metric $g^F$ on $F\to X$, and put
the induced $\Z$-graded Hermitian metric $g^{H^\bullet(Z, F^\pm|_Z)}$ on
$H^\bullet(Z, F^\pm|_Z)\to B$ as in Theorem \ref{thm 3.3.1}.

By (\ref{eq 2.1.14}), the left-hand side of (\ref{eq 3.3.1}) can be written as
\begin{equation}\label{eq 3.3.27}
\begin{split}
&\CS(\wt{f}^*(k\nabla^{H(Z, F^-|_Z), u}\oplus\un{d^N}), k\nabla^{H(Z, F^+|_Z), u}\oplus
\un{d^M})\\
=&\CS(\wt{f_0}^*(k\nabla^{H^{\even}(Z, F^-|_Z), u}\oplus d^N), k\nabla^{H^{\even}(Z,
F^+|_Z), u}\oplus d^M)\\
&\quad-\CS(\wt{f_1}^*(k\nabla^{H^{\odd}(Z, F^-|_Z), u}\oplus d^N), k\nabla^{H^{\odd}
(Z, F^+|_Z), u}\oplus d^M)
\end{split}
\end{equation}
in \dis{\frac{\Omega^{\odd}(B)}{\im(d)}}. Recall from (\ref{eq 3.3.3}) and
(\ref{eq 3.3.18}) that we have the following diagram of smooth bundle isomorphisms
\cdd{kH^{\even}(Z, F^+|_Z)\oplus\C^M @>\wt{f_0}>> kH^{\even}(Z, F^-|_Z)\oplus\C^N \\
@V h^{\even}_+\oplus\id VV @VV h^{\even}_-\oplus\id V \\ \C^{kn^{\even}_+}\oplus
\C^M @>>\cong> \C^{kn^{\even}_-}\oplus\C^N}
Note that $\wt{f_0}^*((h^{\even}_-)^*d^{kn^{\even}_-}\oplus d^N)$ and $(h^{\even}_+)^*
d^{kn^{\even}_+}\oplus d^M$ are trivial connections on $kH^{\even}(Z, F^+|_Z)\oplus\C^M
\to B$. By (\ref{eq 2.1.8}), (\ref{eq 2.1.9}), (\ref{eq 2.1.10}), Remark \ref{remark 2},
and the fact that $\wt{f_0}$ covers the identity map $\id_B$ we have
\begin{equation}\label{eq 3.3.28}
\begin{split}
&\CS(\wt{f_0}^*(k\nabla^{H^{\even}(Z, F^-|_Z), u}\oplus d^N), k\nabla^{H^{\even}(Z,
F^+|_Z), u}\oplus d^M)\\
=&\CS(\wt{f_0}^*(k\nabla^{H^{\even}(Z, F^-|_Z), u}\oplus d^N), \wt{f_0}^*((h^{\even}_-)^*
d^{kn^{\even}_-}\oplus d^N))\\
&\quad+\CS(\wt{f_0}^*((h^{\even}_-)^*d^{kn^{\even}_-}\oplus d^N), k\nabla^{H^{\even}
(Z, F^+|_Z), u}\oplus d^M)\\
=&\CS(k\nabla^{H^{\even}(Z, F^-|_Z), u}\oplus d^N, (h^{\even}_-)^*d^{kn^{\even}_-}
\oplus d^N)\\
&\quad+\CS((h^{\even}_+)^*d^{kn^{\even}_+}\oplus d^M, k\nabla^{H^{\even}(Z, F^+|_Z), u}
\oplus d^M)\\
=&-\CS((h^{\even}_-)^*d^{kn^{\even}_-}, k\nabla^{H^{\even}(Z, F^-|_Z), u})+\CS
((h^{\even}_+)^*d^{kn^{\even}_+}, k\nabla^{H^{\even}(Z, F^+|_Z), u})
\end{split}
\end{equation}
in \dis{\frac{\Omega^{\odd}(B)}{\Omega^{\odd}_\Q(B)}}. By a similar argument we have
\begin{equation}\label{eq 3.3.31}
\begin{split}
&-\CS(\wt{f_1}^*(k\nabla^{H^{\odd}(Z, F^-|_Z), u}\oplus d^N), k\nabla^{H^{\odd}
(Z, F^+|_Z), u}\oplus d^M)\\
=&-\CS((h^{\odd}_+)^*d^{kn^{\odd}_+}, k\nabla^{H^{\odd}(Z, F^+|_Z), u})+
\CS((h^{\odd}_-)^*d^{kn^{\odd}_-}, k\nabla^{H^{\odd}(Z, F^-|_Z), u})
\end{split}
\end{equation}
in \dis{\frac{\Omega^{\odd}(B)}{\Omega^{\odd}_\Q(B)}}. By (\ref{eq 3.3.28}) and
(\ref{eq 3.3.31}), (\ref{eq 3.3.27}) becomes
\begin{equation}\label{eq 3.3.32}
\begin{split}
&\CS(\wt{f}^*(k\nabla^{H(Z, F^-|_Z), u}\oplus\un{d^N}), k\nabla^{H(Z, F^+|_Z), u}\oplus
\un{d^M})\\
=&\CS((h^{\even}_+)^*d^{kn^{\even}_+}, k\nabla^{H^{\even}(Z, F^+|_Z), u})-
\CS((h^{\even}_-)^*d^{kn^{\even}_-}, k\nabla^{H^{\even}(Z, F^-|_Z), u})\\
&\quad-\CS((h^{\odd}_+)^*d^{kn^{\odd}_+}, k\nabla^{H^{\odd}(Z, F^+|_Z), u}+
\CS((h^{\odd}_-)^*d^{kn^{\odd}_-}, k\nabla^{H^{\odd}(Z, F^-|_Z), u})
\end{split}
\end{equation}
in \dis{\frac{\Omega^{\odd}(B)}{\Omega^{\odd}_\Q(B)}}. Therefore (\ref{eq 3.3.1})
becomes
\begin{equation}\label{eq 3.3.33}
\begin{split}
&\CS((h^{\even}_+)^*d^{kn^{\even}_+}, k\nabla^{H^{\even}(Z, F^+|_Z), u})-
\CS((h^{\even}_-)^*d^{kn^{\even}_-}, k\nabla^{H^{\even}(Z, F^-|_Z), u})\\
&\quad-\CS((h^{\odd}_+)^*d^{kn^{\odd}_+}, k\nabla^{H^{\odd}(Z, F^+|_Z), u}+
\CS((h^{\odd}_-)^*d^{kn^{\odd}_-}, k\nabla^{H^{\odd}(Z, F^-|_Z), u})\\
=&\int_{X/B}e(\nabla^{T^VX})\wedge\CS(\wt{j}^*k\nabla^{-, u}, k\nabla^{+, u})
\end{split}
\end{equation}
in \dis{\frac{\Omega^{\odd}(B)}{\Omega^{\odd}_\Q(B)}}. Since $(h^\bullet_\pm)^*
d^{kn^\bullet_\pm}$ is a trivial connection on $kH^\bullet(Z, F^\pm|_Z)\to B$, it follows
from (\ref{eq 3.3.26}) that the left-hand side of (\ref{eq 3.3.33}) is a differential form
respective of left-hand side of (\ref{eq 3.3.25}), and the right-hand side of
(\ref{eq 3.3.33}) is a differential form respective of the right-hand side of
(\ref{eq 3.3.25}). Thus (\ref{eq 3.3.33}) implies that (\ref{eq 3.3.25}) holds.
\end{proof}

Let $F\to X$ be a $\Z_2$-graded complex flat vector bundle with $\Z_2$-graded flat
connection $\nabla^F$. By (\ref{eq 3.3.26}) and Lemma \ref{lemma 3.2.1} we have
\begin{equation}\label{eq 3.3.34}
\begin{split}
\re(\CCS(H(Z, F|_Z), \nabla^{H(Z, F|_Z)}))&=\sum_{k=0}(-1)^k\re(\CCS(H^k(Z, F^+|_Z),
\nabla^{H^k(Z, F^+|_Z)}))\\
&\quad-\sum_{k=0}(-1)^k\re(\CCS(H^k(Z, F^-|_Z), \nabla^{H^k(Z, F^-|_Z)})).
\end{split}
\end{equation}

We now deduce (\ref{eq 1.0.2}) for $\dim(Z)$ even.
\begin{coro}\label{coro 3.3.4}
Let $\pi:X\to B$ be a submersion with closed fibers $Z$ with $\dim(Z)$ even, and $F\to X$
a complex flat vector bundle with flat connection $\nabla^F$. By putting a Hermitian
metric $g^F$ on $F\to X$ and the induced $\Z$-graded Hermitian metric $g^{H(Z, F|_Z)}$ on
$H(Z, F|_Z)\to B$ we have
\begin{equation}\label{eq 3.3.35}
\sum_{k=0}^n(-1)^k\re(\CCS(H^k(Z, F|_Z), \nabla^{H^k(Z, F|_Z)}))=\int_{X/B}e(T^VX)\cup
\re(\CCS(F, \nabla^F))
\end{equation}
in $H^{\odd}(B; \R/\Q)$.
\end{coro}
\begin{proof}
Let $F\to X$ be a complex flat vector bundle with flat connection $\nabla^F$. Write
$\ell=\rk(F)$. Put a Hermitian metric $g^F$ on $F\to X$ and a trivial metric $g^\ell$ on
$\C^\ell\to X$. Choose and fix a trivial connection $d^\ell$ on $\C^\ell\to X$. Define a
$\Z_2$-graded complex vector bundle $\FF\to X$ by $\FF^+=F$ and $\FF^-=\C^\ell$, and a
$\Z_2$-graded connection $\nabla^{\FF}=\nabla^+\oplus\nabla^-$ on $\FF\to X$, where
$\nabla^+=\nabla^F$ and $\nabla^-=d^\ell$. Then $\FF\to X$ is a $\Z_2$-graded complex flat
vector bundle of virtual rank zero with $\Z_2$-graded flat connection $\nabla^{\FF}$, and
is equipped with a $\Z_2$-graded Hermitian metric $g^F\oplus g^\ell$. By (\ref{eq 3.2.10})
and Remark \ref{remark 2} we have
$$\CCS(\FF, \nabla^{\FF})=\CCS(F, \nabla^F)-\CCS(\C^\ell, d^\ell)=\CCS(F, \nabla^F)$$
in $H^{\odd}(B; \C/\Q)$. Thus
\begin{equation}\label{eq 3.3.36}
\re(\CCS(\FF, \nabla^{\FF}))=\re(\CCS(F, \nabla^F))
\end{equation}
in $H^{\odd}(B; \R/\Q)$. Denote by $g^{H(Z, \FF|_Z)}$ the induced $\Z$-graded Hermitian
metric on the $\Z$-graded complex flat vector bundle $H(Z, \FF|_Z)\to B$ with $\Z$-graded
flat connection
$\nabla^{H(Z, \FF|_Z)}$. By applying Corollary \ref{coro 3.3.3} to $(\FF, \nabla^{\FF})$,
(\ref{eq 3.3.25}) and (\ref{eq 3.3.36}) imply
\begin{equation}\label{eq 3.3.37}
\re(\CCS(H(Z, \FF|_Z), \nabla^{H(Z, \FF|_Z)}))=\int_{X/B}e(T^VX)\cup\re(\CCS(F, \nabla^F)).
\end{equation}
By (\ref{eq 3.3.34}), the left-hand side of (\ref{eq 3.3.37}) becomes
\begin{equation}\label{eq 3.3.38}
\begin{split}
&\re(\CCS(H(Z, \FF|_Z), \nabla^{H(Z, \FF|_Z)}))\\
=&\sum_{k=0}(-1)^k\re(\CCS(H^k(Z, F|_Z), \nabla^{H^k(Z, F|_Z)}))\\
&\quad-\sum_{k=0}(-1)^k\re(\CCS(H^k(Z, \C^\ell|_Z), \nabla^{H^k(Z, \C^\ell|_Z)}))\\
=&\sum_{k=0}(-1)^k\re(\CCS(H^k(Z, F|_Z), \nabla^{H^k(Z, F|_Z)}))\\
&\quad-\rk(F)\sum_{k=0}(-1)^k\re(\CCS(H^k(Z, \C|_Z), \nabla^{H^k(Z, \C|_Z)})).
\end{split}
\end{equation}
As argued in \cite[p.614]{MZ08}, the second part of \cite[Theorem 3.12]{B05} implies that
$$\sum_{k=0}^n(-1)^k\re(\CCS(H^k(Z, \C|_Z), \nabla^{H^k(Z, \C|_Z)}))=0$$
in $H^{\odd}(B; \R/\Q)$. Thus (\ref{eq 3.3.38}) becomes
\begin{equation}\label{eq 3.3.39}
\re(\CCS(H(Z, \FF|_Z), \nabla^{H(Z, \FF|_Z)}))=\sum_{k=0}^n(-1)^k\re(\CCS(H^k(Z, F|_Z),
\nabla^{H^k(Z, F|_Z)}))
\end{equation}
in $H^{\odd}(B; \R/\Q)$. It follows from (\ref{eq 3.3.37}) and (\ref{eq 3.3.39}) that
(\ref{eq 3.3.35}) holds.
\end{proof}
\bibliographystyle{amsplain}
\bibliography{MBib}
\end{document}